\documentclass[12pt,a4paper,twoside,reqno]{amsart} 
 \usepackage{amsfonts,amssymb,amscd,amsmath,enumerate,verbatim,calc} 

\usepackage{color,psfrag}

\newtheorem{proposition}{Proposition}[section]
\newtheorem{theorem}[proposition]{Theorem}
\newtheorem{corollary}[proposition]{Corollary}
\newtheorem{definition}[proposition]{Definition}
\newtheorem{lemma}[proposition]{Lemma}

\allowdisplaybreaks

\theoremstyle{remark}

\begin{document}
\title{On a Tur{\'a}n's theorem for arithmetic progressions}
\author{Tokuhon Makoto Minamide, Haruka Sakai,\\
 and\\
 Yoshio Tanigawa}
\thanks{
This work is supported by JSPS KAKENHI Grant No. 22K03245.}
\maketitle
\begin{abstract}
Let $m\geq 1$ be a fixed integer, $a$ an integer satisfying $(a,m)=1$, and $z\geq 1$ a real parameter. Denote by $\omega_{z}(n;m,a)$
the number of distinct prime divisors $p$ of $n$ satisfying $p\equiv a\, (m)$ and $p\leq z$. 
We study an asymptotic behaviour of $\sum_{n\leq x}\left(\omega_{z}(n;m,a)-\frac{1}{\varphi(m)}\log\log z\right)^{k}$ as $x\to\infty$ for a wide range of positive integer $k\geq 2$, where $\varphi(\cdot)$ is the Euler function.
Following a method of Granville and Soundararajan we lead an asymptotic formula for the above. Also, we investigate
$\sum_{n\leq x}\left(\omega(n;m,a)-\frac{1}{\varphi(m)}\log\log x\right)^{k}$, where $\omega(n;m,a)$ denotes the number of distinct prime divisors $p$ of $n$ such that
$p\equiv a\, (m)$.     
\end{abstract}

\section{Introduction}
This notes is a sequel of our research \cite{MST} 
in which we investigated a type of Tur{\'a}n's theorem on $\omega(n)$ which denotes $\sum_{p|n}1$ the number of distinct prime divisors $p$ of the natural number $n$. 
In \cite{turan}, 
the following estimate (\ref{turan-result}) was shown by Tur{\'a}n 
\begin{align}\label{turan-result}
\sum_{n\leq x}(\omega(n)-\log\log x)^{2}=O(x\log\log x),\quad (x\to\infty) 
\end{align}
to simplify the proof of  the Hardy-Ramanujan theorem in \cite{HR} that the arithmetical function $\omega(n)$ has the normal order $\log\log n$. See \cite{turan} and \cite[Ch.~22]{HW}, in detail. 

In \cite{MST}, we have studied a problem on $\omega_{z}(n)$ related to (\ref{turan-result}), which is defined by 
\begin{align}\label{teigi-omega-z-ika}
\omega_{z}(n):=\sum_{\begin{subarray}{c}p|n\\ p\leq z\end{subarray}}1,
\end{align}
where $z\geq 1$ is a real number. 
To introduce a result in \cite{MST}, we define $C_{j}$ and $\tilde{C}_{j}$ for any non-negative integer $j$ as
\begin{align}\label{teigi-C-j-911}
\begin{split}
& C_{0}=1,\\
& C_{j}:=\frac{\Gamma (j+1)}{2^{\frac{j}{2}}\Gamma\left(\frac{j}{2}+1\right)}, \quad \text{and}\quad 
\tilde{C}_{j}:=\frac{\sqrt{2}}{6}\frac{(j-1)\Gamma\left(\frac{j}{2}+1\right)}{\Gamma\left(\frac{j}{2}+\frac{1}{2}\right)}C_{j},\quad (j\geq 1),
\end{split}
\end{align}
where $\Gamma(\cdot)$ denotes the standard gamma function. 
Note that by Stirling formula and Lemma in \cite[p.~58]{Tit}
\begin{align*}
C_{j} \sim \sqrt{2}\left(\frac{j}{e}\right)^{\frac{j}{2}}\gg 2^{j}\quad \text{and}\quad C_{j-1}\asymp\frac{C_{j}}{\sqrt{j}},
\end{align*}
as $j\to\infty$, moreover
\begin{align*}
\tilde{C}_{j} \asymp j^{\frac{3}{2}}C_{j} \quad \text{and}\quad \tilde{C}_{j-1}\asymp\frac{\tilde{C}_{j}}{\sqrt{j}}\asymp jC_{j}.
\end{align*}
For any sequences $\{a_{j}\}$ and $\{b_{j}\}$, if $a_{j}/b_{j}\to 1$ as $j\to\infty$, then we write $a_{j}\sim b_{j}$, and if $a_{j}\ll b_{j}$ and $b_{j}\ll a_{j}$ as $j\to\infty$,
then we write $a_{j}\asymp b_{j}$.

Applying the method of Granville and Soundararajan in \cite{GS}, which is quite new,
we have obtained the following theorem in \cite{MST}.
\begin{theorem}\label{Hatsumi}
Let $x\geq 1$ and $z\geq 1$ be sufficiently large. For positive integers $k\geq 2$, we assume that
$k \leq (\log \log z)^{1/3} $ and  $e^{e^{k^3}}\leq z\leq x^{1/k}$. Then, we have
the following (a) and (b), uniformly in $k$, $z$, and $x$.
\begin{enumerate}
\item[\rm (a)] For even integers $k\geq 2$, we have
\begin{align*}
\sum_{n\leq x}(\omega_{z}(n)-\log\log z)^{k} = C_{k} x(\log \log z)^{\frac{k}{2}} \left( 1+O\left( \frac{k^{3}}{\log\log z}\right) \right).
\end{align*}
\item[\rm (b)] For odd integers $k\geq 3$, we have three formulas (I), (II), and (III) as follows.\\
 
\begin{enumerate}
\item[\rm (I)]If $3\leq k \leq (\log\log z)^{1/7}$, then we have
\begin{align*}
&\sum_{n\leq x}(\omega_{z}(n)-\log\log z)^{k}\\
&=\tilde{C}_{k}x (\log\log z) ^{\frac{k-1}{2}} +B(k^{3/2}C_{k}) x (\log\log z)^{\frac{k-1}{2}} \left(\frac{C_{k-1}}{k^{1/2} C_{k}}\right)\\
&\quad +O\left((k^{3/2}C_{k}) x (\log\log z)^{\frac{k-1}{2}} \cdot \frac{1}{(k\log\log z)^{1/2}}\right).
\end{align*}
\item[\rm (II)] If $(\log\log z)^{1/7}<k <(\log\log z)^{1/4}$, then we have
\begin{align*}
&\sum_{n\leq x}(\omega_{z}(n)-\log\log z)^{k} \\
&=\tilde{C}_{k} x (\log\log z)^{\frac{k-1}{2}} + B \left(k^{3/2}C_{k}\right) x (\log\log z)^{\frac{k-1}{2}} \left(\frac{C_{k-1}}{k^{1/2}C_{k}}\right)\\
&\quad +O\left( \left(k^{3/2}C_{k}\right) x (\log\log z)^{\frac{k-1}{2}}\frac{k^{3}}{\log\log z}\right).
\end{align*}
\item[\rm (III)] If $(\log\log z)^{1/4}\leq k \leq (\log\log z)^{1/3}$, then we have
\begin{align*}
&\sum_{n\leq x}\left(\omega_{z}(n)-\log\log z\right)^{k}\\
&= \tilde{C}_{k} x (\log\log z)^{\frac{k-1}{2}} +O\left(\left(k^{3/2}C_{k}\right)x (\log\log z)^{\frac{k-1}{2}} \frac{k^{3}}{\log\log z}\right).
\end{align*}
\end{enumerate}
\end{enumerate}
Here, $B$ is the constant defined by
\begin{align*}
B:=\lim_{x\to\infty}\left(\sum_{p\leq x}\frac{1}{p}-\log\log x\right) \approx 0.261.
\end{align*}
\end{theorem}

Note that although it is not asymptotic formula which was obtained in \cite{GS} for the odd power moments of $\omega(n)-\log\log x$ (see Theorem \ref{GS-Theorem-1}, below),
however, under the restriction $p\leq z$ in $\omega(n)$ the above asymptotic result (b) of Theorem \ref{Hatsumi} was obtained in the previous literature \cite{MST}.

\bigskip

We now set a main object which is investigated in this notes, definitely.
Let $x\geq 1$ and $z\geq 1$. Moreover, let $m\geq 1$ be a fixed integer, and $a$ an integer which is coprime to $m$. 
Define $\omega_{z}(n;m,a)$ for any natural number $n$ as
\begin{align}\label{go}
\omega_{z}(n;m,a):=\sum_{\begin{subarray}{c} p|n\\ p \leq z \\ p\equiv a\, (m) \end{subarray}}1,
\end{align}
and we shall investigate as $x\to\infty$
\begin{align}
\sum_{n\leq x}\left(\omega_{z}(n;m,a)-\frac{1}{\varphi(m)}\log\log z\right)^{k}\label{moment-problem-1}
\end{align}
for a wide range of positive integers $k\geq 2$, where $\varphi(\cdot)$ denotes the Euler function, as usual.

Regarding this problem for (\ref{moment-problem-1}), we shall prove the following theorem.
\begin{theorem}\label{Natsuki}
Fix a positive integer $m$ and take an integer $a$ such that $(a,m)=1$.
Let $x\geq 1$ and $z\geq 1$ be sufficiently large. Further, let $k$ be integer satisfying $2\leq k \leq \left(\frac{1}{\varphi(m)}\log\log z\right)^{1/3}$,
and assume that $z\leq mx^{1/k}$.
Then, uniformly in $k$, $z$, and $x$ we have the following formulas:
\begin{enumerate}
\item[\rm (a)] if $k\geq 2$ is even, then
\begin{align*}
&\sum_{n\leq x}\left(\omega_{z}(n;m,a)-\frac{1}{\varphi(m)}\log\log z\right)^{k}\\
&=C_{k}x\left(\frac{1}{\varphi(m)}\log\log z\right)^{\frac{k}{2}}\left(1+O_{m}\left(\frac{k^{3}}{\frac{1}{\varphi(m)}\log\log z}\right)\right)
\end{align*}
\item[\rm (b)]if $k\geq 3$ is odd, then
\begin{enumerate}
\item[\rm (i)] in the case $3\leq k \leq \left(\frac{1}{\varphi(m)}\log\log z\right)^{\frac{1}{7}}$,
\begin{align*}
&\sum_{n\leq x}\left(\omega_{z}(n;m,a)-\frac{1}{\varphi(m)}\log\log z\right)^{k}\\
&=(\tilde{C}_{k} +b(m,a)k C_{k-1})x \left(\frac{1}{\varphi(m)}\log\log z\right)^{\frac{k-1}{2}}\\
&\quad +O_{m}\left(\left(k^{3/2}C_{k}\right)x \left(\frac{1}{\varphi(m)}\log\log z\right)^{\frac{k-1}{2}}\frac{1}{k^{1/2}\left(\frac{1}{\varphi(m)}\log\log z\right)^{1/2}}\right),
\end{align*}
\item[\rm (ii)] in the case 
$\left(\frac{1}{\varphi(m)}\log\log z\right)^{\frac{1}{7}}< k <\left(\frac{1}{\varphi(m)}\log\log z\right)^{\frac{1}{4}}$
\begin{align*}
&\sum_{n\leq x}\left(\omega_{z}(n;m,a)-\frac{1}{\varphi(m)}\log\log z\right)^{k}\\
&=(\tilde{C}_{k} +b(m,a)k C_{k-1})x \left(\frac{1}{\varphi(m)}\log\log z\right)^{\frac{k-1}{2}}\\
&\quad +O_{m}\left(\left(k^{3/2}C_{k}\right)x \left(\frac{1}{\varphi(m)}\log\log z\right)^{\frac{k-1}{2}}
\frac{k^{3}}{\frac{1}{\varphi(m)}\log\log z} \right),
\end{align*}
\item[\rm (iii)] in the case $\left(\frac{1}{\varphi(m)}\log\log z\right)^{1/4}\leq k \leq \left(\frac{1}{\varphi(m)}\log\log z\right)^{1/3}$ 
\begin{align*}
&\sum_{n\leq x}\left(\omega_{z}(n;m,a)-\frac{1}{\varphi(m)}\log\log z\right)^{k}\\
&=\tilde{C}_{k}x\left(\frac{1}{\varphi(m)}\log\log z\right)^{\frac{k-1}{2}}\\
&\quad +O_{m}\left(\left(k^{3/2}C_{k}\right) x \left(\frac{1}{\varphi(m)}\log\log z\right)^{\frac{k-1}{2}}\frac{k^{3}}{\frac{1}{\varphi(m)}\log\log z}\right).
\end{align*}
\end{enumerate}
\end{enumerate}
Here,
\begin{align*}
b(m,a):=\lim_{x\to\infty}\left(\sum_{\begin{subarray}{c}p\leq x\\ p\equiv a\, (m) \end{subarray}}\frac{1}{p}-\frac{1}{\varphi(m)}\log\log x\right).
\end{align*}
(See \cite[p.~126, (c)]{MV}.)
\end{theorem}

Recall now the aim in the literature \cite{GS} by Granville and Soundararajan, 
which was to prove the following theorem for $\omega(n)$.
\begin{theorem}[{\cite[p.~17, Theorem 1]{GS}}]\label{GS-Theorem-1}
Let $x\geq 1$ be sufficient large. For any positive integer $k\geq 2$ satisfying $k\leq (\log\log x^{1/k})^{1/3}$,
we have uniformly in $k$ and $x$
\begin{enumerate}
\item[\rm (a)] if $k\geq 2$ is even, then
\begin{align*}
&\sum_{n\leq x}\left(\omega(n)-\log\log x\right)^{k}=C_{k}x\left(\log\log x\right)^{\frac{k}{2}}
\left(1+O\left(\frac{k^{3/2}}{\left(\log\log x\right)^{1/2}}\right)\right),
\end{align*}
\item[\rm (b)]if $k\geq 3$ is odd, then
\begin{align*}
&\sum_{n\leq x}\left(\omega(n)-\log\log x\right)^{k} =O\left(C_{k} x \left(\log\log x\right)^{\frac{k-1}{2}}k^{3/2}\right).
\end{align*}
\end{enumerate}
\end{theorem}

Using a lemma (Lemma \ref{lemma-sakai}, in Section 2, below) for the proof of Theorem \ref{Natsuki} 
we shall also derive a theorem which corresponds to Theorem \ref{GS-Theorem-1}, for
\begin{align}\label{teigi-omega-nma}
\omega(n;m,a):=\sum_{\begin{subarray}{c}p|n\\ p\equiv a\, (m)\end{subarray}}1.
\end{align}
Actually, we shall prove the following theorem.
\begin{theorem}\label{umeko}
Let $m\geq 1$ be any fixed integer, and $a$ an integer satisfying $(a,m)=1$. 
Let $x\geq 1$ be sufficient large. For any positive integer $k\geq 2$ satisfying $k\leq \left(\frac{1}{\varphi(m)}\log\log \left(mx^{1/k}\right)\right)^{1/3}$,
we have uniformly in $k$ and $x$,
\begin{enumerate}
\item[\rm (a)] if $k\geq 2$ is even, then
\begin{align*}
&\sum_{n\leq x}\left(\omega(n;m,a)-\frac{1}{\varphi(m)}\log\log x\right)^{k}\\
&=C_{k}x\left(\frac{1}{\varphi(m)}\log\log x\right)^{\frac{k}{2}}
\left(1+O_{m}\left(\frac{k^{3}}{\frac{1}{\varphi(m)}\log\log x}\right)\right),
\end{align*}
\item[\rm (b)]if $k\geq 3$ is odd, then
\begin{align*}
&\sum_{n\leq x}\left(\omega(n;m,a)-\frac{1}{\varphi(m)}\log\log x\right)^{k}\\
&=O_{m}\left(C_{k} x \left(\frac{1}{\varphi(m)}\log\log x\right)^{\frac{k-1}{2}}k^{3/2}\right).
\end{align*}
\end{enumerate}
Here, the notation $O_{m}$ means that the implied $O$-constant depends on the fixed $m$. 
\end{theorem}

Throughout this notes, we denote by $\pi(z;m,a)$ the number of primes $p\leq z$ satisfying $p\equiv a\, (m)$. 
Also, we might use $\ll_{m}$ instead of $O_{m}$, also we write $O$ and $\ll$ briefly.
%
%
%
%
%
%
\section{Lemma}
To prove Theorems \ref{Natsuki} and \ref{umeko} by the method in \cite{GS}, we shall prepare Lemma \ref{lemma-sakai}, below, which is corresponding to
Proposition 2 in \cite[p.~17]{GS} by Granville and Soundararajan. To this end we shall recall first the definition of $f_{p}(n)$ which was introduced in \cite{GS}.
See also \cite{MoSo} by Montgomery and Soundararajan.

\begin{definition}[{\cite[p.~17, 18]{GS}}]\label{teigi-fpn-aaa}
For any natural number $n$ and any prime $p$, we define $f_{p}(n)$ by
\begin{align*}
f_{p}(n):=
\begin{cases}
1-\frac{1}{p} & (\textit{if}\ p|n)\\
-\frac{1}{p}  & (\textit{if}\ p\nmid n)
\end{cases}.
\end{align*}
Moreover, for any primes $p_{1},\ldots, p_{l}$ (these are not necessarily to distinct) we write
\begin{align*}
f_{p_{1}\cdots p_{l}}(n):=f_{p_{1}}(n)\cdots f_{p_{l}}(n).
\end{align*}
\end{definition}

We shall restrict in \cite[Proposition 2]{GS} the sum over all primes $p\leq z$ that primes $p$'s are congruent to $a$ to modulus $m$ and obtain the following lemma.
%
%
%
\begin{lemma}[{cf. \cite[Proposition 2]{GS}}]\label{lemma-sakai}
Let $m\geq 1$ be a fixed integer and $a$ an integer such that $(a,m)=1$. Moreover, let $x\geq 1$ and $z\geq 1$ be sufficiently large number and
a positive integer $k$ is subject to $k\leq \left(\frac{1}{\varphi(m)}\log\log z\right)^{1/3}$. Then, we have uniformly in $k$, $z$, and $x$
\begin{enumerate}
\item[\rm (a)] for even integers $j$ ($2\leq j\leq k$),
\begin{align*}
\sum_{n\leq x}\left(\sum_{\begin{subarray}{c}p\leq z \\ p\equiv a\, (m)\end{subarray}}f_{p}(n)\right)^{j}
&=C_{j}x\left(\frac{1}{\varphi(m)}\log\log z\right)^{\frac{j}{2}}\left(1+O_{m}\left(\frac{j^{3}}{\frac{1}{\varphi(m)}\log\log z}\right)\right)\\
&\quad +O(2^{j}\pi\left(z;m,a\right)^{j}), 
\end{align*}
\item[(b)] for odd integers $j$ ($3\leq j \leq k$)
\begin{align*}
\sum_{n\leq x}\left(\sum_{\begin{subarray}{c}p\leq z\\ p\equiv a\, (m)\end{subarray}}f_{p}(n)\right)^{j}
&=\tilde{C}_{j}x \left(\frac{1}{\varphi(m)}\log\log z\right)^{\frac{j-1}{2}}\left(1+O_{m}\left(\frac{j^{3}}{\frac{1}{\varphi(m)}\log\log z}\right)\right)\\
&\quad +O\left(2^{j}\pi(z;m,a)^{j}\right),
\end{align*}
\item[(c)] for $j=0,1$, obviously
\begin{align*}
\sum_{n\leq x}\left(\sum_{\begin{subarray}{c}p\leq z\\ p\equiv a\, (m)\end{subarray}}f_{p}(n)\right)^{j}
=\begin{cases}
O(x) & (j=0)\\
O(\pi(z;m,a)) & (j=1)
\end{cases}.
\end{align*}
\end{enumerate}
\end{lemma}
To show Lemma \ref{lemma-sakai}, we recall the function $G(r)$ which is introduced in \cite{GS} and its properties.
\begin{definition}[{\cite[p.~18]{GS}}]\label{teigi-fpn-bbb}
For any primes $p_{1},\ldots, p_{l}$, we put 
\begin{align*}
r:=p_{1}\cdots p_{l}=\prod_{i=1}^{s}q_{i}^{\alpha_{i}},
\end{align*}
where the right-hand side of the above product denotes the prime factorization of $r$. Further, for the above $r$ we write
\begin{align*}
R:=\prod_{i=1}^{s}q_{i},
\end{align*}
which is so called the kernel of $r$. Note that if $(n,R)=d$, then it holds that
\begin{align}\label{nature-fpn-1}
f_{p_{1}\cdots p_{l}}(n)=f_{p_{1}\cdots p_{l}}(d).
\end{align}
Using the function $f_{p}(n)$ in Definition \ref{teigi-fpn-aaa} and the Euler function $\varphi(n)$ we define $G(r)$ for any $r=p_{1}\cdots p_{l}$ as follows.
\begin{align}\label{teigi-Gr}
G(r):=\frac{1}{R}\sum_{d|R}f_{p_{1}\cdots p_{l}}(d)\varphi\left(\frac{R}{d}\right).
\end{align}
\end{definition}
The following properties of $G(r)$ are used in the proof of Lemma \ref{lemma-sakai}, which are stated in \cite{GS} and quoted in \cite{MST}.
\begin{lemma}[{\cite[p.~18, 19]{GS}}]\label{lemma-Gr}
Let $r$ and $r^{\prime}$ be positive integers $\geq 2$. Then we have the following properties of $G(r)$.
\begin{enumerate}
\item[\rm (i)] If $(r,r^{\prime})=1$, then $G\left(rr^{\prime}\right)=G(r)G\left(r^{\prime}\right)$.
\item[\rm (ii)] For the prime factorization of $r$, if $r=\prod_{q^{\alpha}|| r}q^{\alpha}$, then
\begin{align*}
G(r)=\prod_{q^{\alpha} ||r} \left(\frac{1}{q}\left(1-\frac{1}{q}\right)^{\alpha}+\left(-\frac{1}{q}\right)^{\alpha} \left(1-\frac{1}{q}\right)\right).
\end{align*}
\item[\rm (iii)] If $r$ is not square-full (that is, the exponent of some prime is one), then $G(r)=0$.
\item[\rm (iv)] For any prime $q$ and any integer $\alpha \geq 2$, we have
\begin{align*}
0\leq G\left(q^{\alpha}\right) \leq \frac{1}{q}\left(1-\frac{1}{q}\right).
\end{align*}
\end{enumerate}
\end{lemma}
We shall now prove Lemma \ref{lemma-sakai}.
\begin{proof}[Proof of Lemma \ref{lemma-sakai}]
It is trivial for the case $j=0$. As for the case $j=1$, by the definition of $f_{p}(n)$ (Definition \ref{teigi-fpn-aaa}) we have
\begin{align*}
\sum_{n\leq x}\left(\sum_{\begin{subarray}{c}p\leq z\\ p\equiv a\, (m)\end{subarray}}f_{p}(n)\right)
&=\sum_{\begin{subarray}{c}p\leq z\\ p\equiv a\, (m) \end{subarray}}\left(\sum_{\begin{subarray}{c}n\leq x\\ p|n\end{subarray}}\left(1-\frac{1}{p}\right)
+\sum_{\begin{subarray}{c}n\leq x\\ p\nmid n\end{subarray}}\left(-\frac{1}{p}\right)\right)\\
&=\sum_{\begin{subarray}{c}p\leq z\\ p\equiv a\, (m) \end{subarray}}\left(\left[\frac{x}{p}\right]-\frac{x}{p}+O\left(\frac{1}{p}\right)\right)\\
&=O\left(\pi(z;m,a)\right).
\end{align*}
As for the case $2\leq j\leq k$, using the notation $R$ in Definition \ref{teigi-fpn-bbb} and  noting (\ref{nature-fpn-1}) 
we have as the arguments in \cite[p.~18]{GS}
\begin{align}
\sum_{n\leq x}\left(\sum_{\begin{subarray}{c}p\leq z\\ p\equiv a\, (m)\end{subarray}}f_{p}(n)\right)^{j}
&=\sum_{\begin{subarray}{c}p_{1}, \ldots, p_{j}\leq z\\ p_{i}\equiv a\, (m)\end{subarray}}\sum_{n\leq x}f_{p_{1}\cdots p_{j}}(n)\nonumber \\
&=\sum_{\begin{subarray}{c}p_{1},\ldots, p_{j}\leq z\\ p_{i}\equiv a\, (m)\end{subarray}}\left(\sum_{d|R}f_{p_{1}\cdots p_{j}}(d)
  \sum_{\begin{subarray}{c}n\leq x\\ (n,R)=d\end{subarray}}1\right) . \label{mishima}
\end{align}
By using the M{\"o}bius function $\mu(\cdot)$ we see that
\begin{align*}
\sum_{\begin{subarray}{c}n\leq x \\ (n,R)=d\end{subarray}}1 =\sum_{\delta |\frac{R}{d}}\mu(\delta)\left[\frac{x}{d\delta}\right]
=\frac{x}{R}\varphi\left(\frac{R}{d}\right)+O\left(\sum_{\delta|\frac{R}{d}}|\mu(\delta)|\right).
\end{align*}
We apply this to (\ref{mishima}) and use the function $G(r)$ defined in (\ref{teigi-Gr}). Then, we get
\begin{align}\label{ao-blue}
\sum_{n\leq x}\left(\sum_{\begin{subarray}{c}p\leq z \\ p\equiv a\, (m)\end{subarray}}f_{p}(n)\right)^{j}
=x\sum_{\begin{subarray}{c}p_{1}, \ldots, p_{j}\leq z\\ p_{i}\equiv a\, (m)\\ p_{1}\cdots p_{j}: \textit{square-full}\end{subarray}}G(p_{1}\cdots p_{j})+O\left(2^{j}\pi(z;m,a)^{j}\right).
\end{align}
Here, the $O$-term is derived from the bound:
\begin{align*}
\sum_{d|R}|f_{p_{1}\cdots p_{j}}(d)|\sum_{\delta|\frac{R}{d}}|\mu(\delta)|
&\leq \sum_{d|R}\left(\prod_{\begin{subarray}{c}p_{i}\\ p_{i}|d\end{subarray}}1\prod_{\begin{subarray}{c}p_{i}\\ p_{i}\nmid d\end{subarray}}\frac{1}{2}\right)
       2^{\omega(R)-\omega(d)}\\
&\leq \sum_{d|R}\left(\frac{1}{2}\right)^{\omega(R)-\omega(d)}2^{\omega(R)-\omega(d)}\leq 2^{j}.
\end{align*}
To consider the sum of $G(p_{1}\cdots p_{j})$ in (\ref{ao-blue}), noting (iii) of Lemma \ref{lemma-Gr} we shall split it as follows:
\begin{align}
\sum_{\begin{subarray}{c}p_{1},\ldots, p_{j}\leq z \\ p_{i}\equiv a\, (m)\end{subarray}}G(p_{1}\cdots p_{j})&=\sum_{\begin{subarray}{c}p_{1},\ldots, p_{j}\leq z \\ p_{i}\equiv a\, (m)\\ p_{1}\cdots p_{j}: \textit{square-full}\end{subarray}}G(p_{1}\cdots p_{j})\nonumber \\
&=\sum_{1\leq s \leq \frac{j}{2}}\sum_{\begin{subarray}{c}q_{1}<\cdots <q_{s}\leq z\\ q_{i}\equiv a\, (m)\end{subarray}}
  \sum_{\begin{subarray}{c}\alpha_{1}+\cdots +\alpha_{s}=j\\ \alpha_{i}\geq 2\end{subarray}}\frac{j!}{\alpha_{1}!\cdots \alpha_{s}!}G\left(q_{1}^{\alpha_{1}}\cdots q_{s}^{\alpha_{s}}\right)\nonumber\\
&=:\sum_{s\leq \frac{j}{2}}I_{s}(j;m,a)\quad (say)\nonumber\\
&=\begin{cases}
   I_{\frac{j}{2}}(j;m,a)+\sum_{s\leq \frac{j-2}{2}}I_{s}(j;m,a) & (j\geq 2, even)\\
   I_{\frac{j-1}{2}}(j;m,a)+\sum_{s\leq \frac{j-3}{2}}I_{s}(j;m,a) & (j\geq 3, odd)
\end{cases}.\label{Gr-100}
\end{align}
First, in the case of odd $j\geq 3$ we shall reveal the asymptotic formula for (\ref{Gr-100}).

\noindent{\bf (I)} Let $j\geq 3$ be odd. By (i) of Lemma \ref{lemma-Gr} we have in (\ref{Gr-100})
\begin{align*}
I_{\frac{j-1}{2}}(j;m,a)=\sum_{\begin{subarray}{c}q_{1}<\cdots <q_{\frac{j-1}{2}}\leq z\\ q_{i}\equiv a\, (m)\end{subarray}}
                         \frac{j!}{3!2^{\frac{j-1}{2}-1}}
                         \sum_{i=1}^{\frac{j-1}{2}}
                         G\left(q_{i}^{3}\right)\prod_{\begin{subarray}{c}u=1\\ u\neq i\end{subarray}}^{\frac{j-1}{2}}G\left(q_{u}^{2}\right).
\end{align*}
Here, note that
\begin{align*}
0\leq G\left(p^{3}\right)\leq G\left(p^{2}\right)
\end{align*}
by (ii), (iv) of Lemma \ref{lemma-Gr}, and recall the definition of $\tilde{C}_{j}$ in (\ref{teigi-C-j-911}). Then, we see that
\begin{align}\label{iced-blend-pengin}
I_{\frac{j-1}{2}}(j;m,a)
\begin{cases}
\leq {\displaystyle \tilde{C}_{j}\sum_{\begin{subarray}{c}q_{1},\ldots, q_{\frac{j-1}{2}}\leq z\\ distinct \\ q_{i}\equiv a\, (m)\end{subarray}}\prod_{u=1}^{\frac{j-1}{2}}
                                G\left(q_{u}^{2}\right)}\\
\geq {\displaystyle \tilde{C_{j}}\sum_{\begin{subarray}{c}q_{1},\ldots, q_{\frac{j-1}{2}}\leq z\\ distinct \\ q_{i}\equiv a\, (m)\end{subarray}}\prod_{u=1}^{\frac{j-1}{2}}
                                           G\left(q_{u}^{3}\right)}
\end{cases}.
\end{align}
We apply the following formula (\cite[p.~126, (c)]{MV})
\begin{align}\label{mertens-ma}
\sum_{\begin{subarray}{c}p\leq z\\ p \equiv a\, (m)\end{subarray}}\frac{1}{p}
=\frac{1}{\varphi(m)}\log\log z +b(m,a)+O_{m}\left(\frac{1}{\log z}\right)
\end{align}
to $G\left(q_{u}^{2}\right)$ in the former in (\ref{iced-blend-pengin}). Easily, we observe that
\begin{align}
I_{\frac{j-1}{2}}(j;m,a)&\leq \tilde{C}_{j}\left(\sum_{\begin{subarray}{c}p\leq z\\ p \equiv a\, (m)\end{subarray}}\frac{1}{p}\left(1-\frac{1}{p}\right)\right)^{\frac{j-1}{2}}\nonumber \\
&=\tilde{C_{j}}\left(\frac{1}{\varphi(m)}\log\log z\right)^{\frac{j-1}{2}}\left(1+O_{m}\left(\frac{1}{\frac{1}{\varphi(m)}\log\log z}\right)\right)^{\frac{j-1}{2}}. \label{I-ue-1}
\end{align}
To deduce a lower bound for (\ref{iced-blend-pengin}) we shall denote by $p_{m,a}^{*}$ the least prime $\geq 5$ satisfying $q\equiv a\, (m)$, and by $\pi_{l}$
the $l$th smallest prime in $\{p\leq z\,|\, p\equiv a\, (m)\}$. Also, we write $\pi_{l^{*}}=p_{m,a}^{*}$. 
As for the latter in (\ref{iced-blend-pengin}), 
we have 
\begin{align}
I_{\frac{j-1}{2}}(j;m,a)&\geq \tilde{C}_{j}\sum_{\begin{subarray}{c}p_{m,a}^{*}\leq q_{1},\ldots, q_{\frac{j-1}{2}}\leq z\\ distinct\\ q_{u}\equiv a\, (m)\end{subarray}}
                                           \prod_{u=1}^{\frac{j-1}{2}}G\left(q_{u}^{3}\right)\nonumber \\
                       &=\tilde{C}_{j}\sum_{\begin{subarray}{c}p_{m,a}^{*}\leq q_{1},\ldots, q_{\frac{j-1}{2}}\leq z\\ distinct\\ q_{u}\equiv a\, (m)\end{subarray}}
                                            \prod_{u=1}^{\frac{j-1}{2}-1}G\left(q_{u}^{3}\right)
                                            \sum_{\begin{subarray}{c}p_{m,a}^{*}\leq p \leq z\\ p\not = q_{u} \left(u=1,\ldots, \frac{j-1}{2}-1\right)\\ p\equiv a\, (m)\end{subarray}}G\left(p^{3}\right). \label{standardblend-1}
\end{align}
And since $G(p^{3})$ is decreasing for increasing $p\geq 5$ we observe that
\begin{align}
\textit{RHS of (\ref{standardblend-1})} &\geq \tilde{C}_{j}\left(\sum_{\pi_{l^{*}+\frac{j-3}{2}}\leq p \leq z}G(p^{3})\right)^{\frac{j-1}{2}}\nonumber \\
&\geq \tilde{C}_{j}\left(\sum_{\begin{subarray}{c}p\leq z\\ p\equiv a\, (m)\end{subarray}}G\left(p^{3}\right)-\sum_{\begin{subarray}{c}p\leq \pi_{l^{*}+\frac{j-3}{2}}\\ p\equiv a\, (m) \end{subarray}}G\left(p^{3}\right)\right)^{\frac{j-1}{2}}\nonumber \\
&=\tilde{C}_{j}\left(\frac{1}{\varphi(m)}\log\log z +O_{m}(1)\right)^{\frac{j-1}{2}}\nonumber \\
&=\tilde{C}_{j}\left(\frac{1}{\varphi(m)}\log\log z\right)^{\frac{j-1}{2}}\left(1+O_{m}\left(\frac{j}{\frac{1}{\varphi(m)}\log\log z}\right)\right)^{\frac{j-1}{2}}.\label{I-shita-1}\end{align}
By the assumption $k\leq \left(\frac{1}{\varphi(m)}\log\log z\right)^{1/3}$ we find that
\begin{align}
\left|\left(1+O_{m}\left(\frac{j}{\frac{1}{\varphi(m)}\log\log z}\right)\right)^{\frac{j-1}{2}}-1\right|
&\leq \frac{M_{m}j^{2}}{\frac{1}{\varphi(m)}\log\log z}\sum_{l=1}^{\frac{j-1}{2}}\frac{1}{l!}\left(\frac{M_{m}j^{2}}{\frac{1}{\varphi(m)}\log\log z}\right)^{l-1}\nonumber\\
&\leq \frac{M_{m}j^{2}}{\frac{1}{\varphi(m)}\log\log z}\sum_{l=1}^{\infty}\frac{1}{(l-1)!}, \label{I-shita-2}
\end{align} 
where $M_{m}$ is a positive constant depending on $m$. Therefore collecting (\ref{iced-blend-pengin}), (\ref{I-ue-1}), (\ref{I-shita-1}), and (\ref{I-shita-2}),
we obtain that
\begin{align}\label{I-odd-main}
I_{\frac{j-1}{2}}(j;m,a)=\tilde{C}_{j}\left(\frac{1}{\varphi(m)}\log\log z\right)^{\frac{j-1}{2}}\left(1+O_{m}\left(\frac{j^{2}}{\frac{1}{\varphi(m)}\log\log z}\right)\right).
\end{align}
We shall bound the remainder portion $\sum_{s\leq \frac{j-3}{2}}I_{s}(j;m,a)$ in (\ref{Gr-100}) (note that it is an empty sum if $j=3$). Since $\left|G\left(q_{i}^{\alpha_{i}}\right)\right|\leq \frac{1}{q_{i}}$ (by Lemma \ref{lemma-Gr}) we observe that
\begin{align}
\sum_{s\leq \frac{j-3}{2}}I_{s}(j;m,a)&\leq \sum_{s\leq \frac{j-3}{2}}\frac{j!}{2^{s}}\sum_{\begin{subarray}{c}q_{1}<\cdots <q_{s}\leq z\\ q_{i}\equiv a\, (m)\end{subarray}}\frac{1}{q_{1}\cdots q_{s}}\sum_{\begin{subarray}{c}(\alpha_{1}-1)+\cdots +(\alpha_{s}-1)=j-s\\ \alpha_{i}\geq 2\end{subarray}}1\nonumber \\
&\leq \sum_{s\leq \frac{j-3}{2}}\frac{j!}{2^{s}s!}\left(\sum_{\begin{subarray}{c}p\leq z \\ p\equiv a\, (m)\end{subarray}}\frac{1}{p}\right)^{s}\binom{j-s}{s}\nonumber \\
&=\sum_{s\leq \frac{j-3}{2}}C_{j}\frac{\Gamma\left(\frac{j}{2}+1\right)2^{\frac{j}{2}-s}}{s!}\left(\sum_{\begin{subarray}{c}p\leq z \\ p\equiv a\, (m)\end{subarray}}\frac{1}{p}\right)^{s} \frac{(j-s)!}{s!(j-2s)!}\label{tara-1}
\end{align}
Here, we shall remark that by the assumption $k\leq \left(\frac{1}{\varphi(m)}\log\log z\right)^{1/3}$
\begin{align*}
\left(\sum_{\begin{subarray}{c}p\leq z\\ p\equiv a\, (m)\end{subarray}}\frac{1}{p}\right)^{s}
&\leq \left(\frac{1}{\varphi(m)}\log\log z\right)^{s}\left(1+\frac{M_{m}}{\frac{1}{\varphi(m)}\log\log z}\right)^{s}\\
&\leq \left(\frac{1}{\varphi(m)}\log\log z\right)^{s} \sum_{l=0}^{s}\frac{1}{l!}\left(\frac{M_{m}k}{\frac{1}{\varphi(m)}\log\log z}\right)^{l}\\
&\leq 3 \left(\frac{1}{\varphi(m)}\log\log z\right)^{s}. 
\end{align*}
The right-hand side of (\ref{tara-1}), by putting $l=\frac{j-3}{2}-s$, is bounded as
\begin{align}
\textit{RHS of (\ref{tara-1})}&\ll C_{j}\left(\frac{1}{\varphi(m)}\log\log z\right)^{\frac{j-3}{2}}\times\nonumber\\
&\quad \times  \sum_{l=0}^{\frac{j-3}{2}-1}   \frac{\Gamma\left(\frac{j}{2}+1\right)2^{l+\frac{3}{2}}\left(\frac{j+3}{2}+l\right)!}{\left(\frac{j-3}{2}-l\right)!\left(\frac{j-3}{2}-l\right)!(2l+3)!}\frac{1}{\left(\frac{1}{\varphi(m)}\log\log z\right)^{l}}. \label{tara-2}
\end{align}
Here, note that
\begin{align*}
 \Gamma \left(\frac{j}{2}+1\right) \asymp \frac{\Gamma\left(\frac{j+3}{2}\right)}{\sqrt{j}},\quad
 \frac{\Gamma\left(\frac{j+3}{2}\right)}{\left(\frac{j-3}{2}-l\right)!} \ll j^{2} \left(\frac{j}{2}\right)^{l},\quad
& \frac{\left(\frac{j+3}{2}+l\right)!}{\left(\frac{j-3}{2}-l\right)!}\ll j^{3}2^{2l}.
\end{align*}
Therefore , we have
\begin{align*}
\textit{RHS of (\ref{tara-2})}&\ll j^{9/2} C_{j}\left(\frac{1}{\varphi(m)}\log\log z\right)^{\frac{j-3}{2}}\sum_{l=0}^{\frac{j-5}{2}}\frac{1}{(2l+3)!}\left(\frac{j^{3}}{\frac{1}{\varphi(m)}\log\log z}\right)^{l}\\
&\ll \left(j^{3/2}C_{j}\right)\left(\frac{1}{\varphi(m)}\log\log z\right)^{\frac{j-3}{2}} j^{3},
\end{align*}
that is,
\begin{align}\label{tara-3}
\sum_{s\leq \frac{j-3}{2}}I_{s}(j;m,a)\ll \tilde{C}_{j}\left(\frac{1}{\varphi(m)}\log\log z\right)^{\frac{j-3}{2}}j^{3}.
\end{align}
By (\ref{Gr-100}), (\ref{I-odd-main}), and (\ref{tara-3}) we reach the following formula,
\begin{align*}
\sum_{\begin{subarray}{c}p_{1},\ldots, p_{j}\leq z\\ p_{i}\equiv a\, (m)\end{subarray}}G(p_{1}\cdots p_{j})
=\tilde{C}_{j}\left(\frac{1}{\varphi(m)}\log\log z\right)^{\frac{j-1}{2}}\left(1+O_{m}\left(\frac{j^{3}}{\frac{1}{\varphi(m)}\log\log z}\right)\right).
\end{align*}
We now use this in (\ref{ao-blue}), then we get the assertion (b) of Lemma \ref{lemma-sakai} (odd $j\geq 3$).

\bigskip

\noindent{\bf (II)} Let $j\geq 2$ be even. 
First, by the assumption $k\leq \left(\frac{1}{\varphi(m)}\log\log z\right)^{1/3}$ we note that
\begin{align}\label{tara-books}
\left(1+O_{m}\left(\frac{j}{\frac{1}{\varphi(m)}\log\log z}\right)\right)^{\frac{j}{2}}=1+O_{m}\left(\frac{j^{2}}{\frac{1}{\varphi(m)}\log\log z}\right)
\end{align}
as we have observed in (\ref{I-shita-2}). 
In (\ref{Gr-100}), we have
\begin{align*}
I_{\frac{j}{2}}(j;m,a)=C_{j}\sum_{\begin{subarray}{c}q_{1},\ldots, q_{\frac{j}{2}}\leq z\\ distinct \\ q_{i}\equiv a\, (m)\end{subarray}}\prod_{u=1}^{\frac{j}{2}}G(q_{u}^{2}),
\end{align*}
and
\begin{align}\label{tara-4}
I_{\frac{j}{2}}(j;m,a)\leq C_{j}\left(\frac{1}{\varphi(m)}\log\log z\right)^{\frac{j}{2}}\left(1+O_{m}\left(\frac{1}{\frac{1}{\varphi(m)}\log\log z}\right)\right)^{\frac{j}{2}}.
\end{align}
Note that $G(p^{2})$ is decreasing for increasing $p\geq 2$. We take the prime $p_{m,a}^{*}$ which is the least prime satisfying $q\equiv a\, (m)$. 
As in the above argument (I),
we write $\pi_{l^{*}}=p_{m,a}^{*}$. We observe that
\begin{align}
I_{\frac{j}{2}}(j;m,a)& \geq C_{j}\left(\sum_{\begin{subarray}{c}\pi_{l^{*}+\frac{j-2}{2}}\leq p \leq z\\ p\equiv a\, (m)\end{subarray}}G\left(p^{2}\right)\right)^{\frac{j}{2}}\nonumber\\
&\geq C_{j}\left(\sum_{\begin{subarray}{c}p\leq z\\ p\equiv a\, (m)\end{subarray}}\frac{1}{p}+O_{m}(j)\right)^{\frac{j}{2}}\nonumber \\
&= C_{j}\left(\frac{1}{\varphi(m)}\log\log z\right)^{\frac{j}{2}}\left(1+O_{m}\left(\frac{j}{\frac{1}{\varphi(m)}\log\log z}\right)\right)^{\frac{j}{2}}.\label{tara-5}
\end{align}
From (\ref{tara-books}), (\ref{tara-4}), and (\ref{tara-5}) we have
\begin{align}\label{I-main-even}
I_{\frac{j}{2}}(j;m,a)=C_{j}\left(\frac{1}{\varphi(m)}\log\log z\right)^{\frac{j}{2}}\left(1+O_{m}\left(\frac{j^{2}}{\frac{1}{\varphi(m)}\log\log z}\right)\right).
\end{align}
Next, as we have seen in (\ref{tara-1}) and (\ref{tara-2}) we can bound $\sum_{s\leq \frac{j-2}{2}}I_{s}(j;m,a)$ in (\ref{Gr-100}) as follows 
(note that the sum is an empty sum if $j=2$):
\begin{align}
\sum_{s\leq \frac{j-2}{2}}I_{s}(j;m,a)& \ll C_{j}\left(\frac{1}{\varphi(m)}\log\log z\right)^{\frac{j-2}{2}}\times\nonumber\\
&\quad \times \sum_{s\leq \frac{j-2}{2}}\frac{\Gamma\left(\frac{j}{2}+1\right)2^{\frac{j}{2}-s}(j-s)!}{s!s!(j-2s)!}\frac{1}{\left(\frac{1}{\varphi(m)}\log\log z\right)^{\frac{j-2}{2}-s}}\nonumber\\
&= C_{j}\left(\frac{1}{\varphi(m)}\log\log z\right)^{\frac{j-2}{2}}\times\nonumber\\
&\quad \times \sum_{l=0}^{\frac{j-2}{2}-1}\frac{\left(\frac{j}{2}\right)! 2^{l+1}\left(\frac{j}{2}+1+l\right)!}{(\frac{j-2}{2}-l)! \left(\frac{j-2}{2}-l\right)!(2l+2)!}\frac{1}{\left(\frac{1}{\varphi(m)}\log\log z\right)^{l}}. \label{kusai-na-tonari}
\end{align}
We use here
\begin{align*}
\frac{\left(\frac{j}{2}\right)!2^{l+1}}{\left(\frac{j-2}{2}-l\right)!}\leq j\cdot j^{l}\quad and \quad \frac{\left(\frac{j}{2}+1+l\right)!}{\left(\frac{j-2}{2}-l\right)!}\ll j^{2}\cdot j^{2l} 
\end{align*}
to bound (\ref{kusai-na-tonari}) and obtain
\begin{align}\label{kyo-nani-tsukutta}
\sum_{s\leq \frac{j-2}{2}}I(j;m,a) \ll C_{j}\left(\frac{1}{\varphi(m)}\log\log z\right)^{\frac{j-2}{2}}j^{3}.
\end{align}
Collecting results of (\ref{Gr-100}), (\ref{I-main-even}), and (\ref{kyo-nani-tsukutta}) we obtain
\begin{align*}
\sum_{\begin{subarray}{c}p_{1},\ldots, p_{j}\leq z\\ p_{i}\equiv a\, (m)\end{subarray}}G(p_{1}\cdots p_{j})
=C_{j}\left(\frac{1}{\varphi(m)}\log\log z\right)^{\frac{j}{2}}\left(1+O_{m}\left(\frac{j^{3}}{\frac{1}{\varphi(m)}\log\log z}\right)\right).
\end{align*}
Hence, by this and (\ref{ao-blue}) we get the assertion (a) of Lemma \ref{lemma-sakai}. 
\end{proof}

\begin{corollary}\label{kei-lemma-sakai}
Keep the setting in Lemma \ref{lemma-sakai}.
Further, we assume that $z\leq mx^{1/k}$. Then, for $0\leq j\leq k$, uniformly in $k$, $z$, and $x$ we have
\begin{align*}
\sum_{n\leq x}\left(\sum_{\begin{subarray}{c}p\leq z\\ p \equiv a\, (m)\end{subarray}}f_{p}(n)\right)^{j}
\ll_{m} C_{j}x \left(\frac{1}{\varphi(m)}\log\log z\right)^{\frac{j}{2}}.
\end{align*}
\end{corollary}
\begin{proof}
By the assumption $z\leq mx^{1/k}$ in Lemma \ref{lemma-sakai} we have
\begin{align*}
2^{j}\pi(z;m,a)^{j}\ll C_{j}\left(\frac{z}{m}\right)^{j}\ll C_{j}x.
\end{align*}
Then, we get the assertion for even integer $j\geq 0$, immediately. For odd integer $j\geq 1$ in Lemma \ref{lemma-sakai}
we observe that
\begin{align*}
\sum_{n\leq x}\left(\sum_{\begin{subarray}{c}p\leq z\\ p\equiv a\, (m)\end{subarray}}f_{p}(n)\right)^{j}
&\ll_{m} \tilde{C}_{j}\left(\frac{1}{\varphi(m)}\log\log z\right)^{\frac{j-1}{2}}+C_{j}x\\
&\ll_{m} j^{\frac{3}{2}}C_{j}\left(\frac{1}{\varphi(m)}\log\log z\right)^{\frac{j}{2}}\frac{1}{\left(\frac{1}{\varphi(m)}\log\log z\right)^{\frac{1}{2}}}\\
&\ll_{m} C_{j}\left(\frac{1}{\varphi(m)}\log\log z\right)^{\frac{j}{2}}.
\end{align*}
\end{proof}
%
%
%
\section{Proof of Theorem \ref{Natsuki}}
Let $m\geq 1$ be a fixed integer, $a$ an integer satisfying $(a,m)=1$. Regarding $\omega_{z}(n;m,a)$ which is introduced in (\ref{go}),
we shall derive the main theorem (Theorem \ref{Natsuki}) in this notes from Lemma \ref{lemma-sakai}.
Further, as we have discussed Theorem \ref{Hatsumi} for any fixed integer $k\geq 2$ in the previous our notes \cite{MST}, we shall here consider
Theorem \ref{Natsuki} for any fixed integer $k\geq 2$. 

First, we shall give an expression for $\omega_{z}(n;m,a)$, which follows from an idea in \cite{GS}. By Definition \ref{teigi-fpn-aaa} of $f_{p}(n)$ and the formula (\ref{mertens-ma}).
Observe that
\begin{align}
\omega_{z}(n;m,a)&=\sum_{\begin{subarray}{c}p|n\\ p\leq z \\p\equiv a\,(m)\end{subarray}}\left(1-\frac{1}{p}+\frac{1}{p}\right)\nonumber \\
                 &=\sum_{\begin{subarray}{c}p|n\\ p\leq z\\ p\equiv a\, (m)\end{subarray}}\left(1-\frac{1}{p}\right)+\sum_{\begin{subarray}{c}p\leq z\\ p\equiv a\, (m)\end{subarray}}\frac{1}{p}-\sum_{\begin{subarray}{c}p\nmid n\\ p\leq z\\ p\equiv a\, (m)\end{subarray}}\frac{1}{p}\nonumber \\
&=\sum_{\begin{subarray}{c}p\leq z\\ p\equiv a\, (m)\end{subarray}}f_{p}(n) +\frac{1}{\varphi(m)}\log\log z +b(m,a)+O_{m}\left(\frac{1}{\log z}\right). \label{cococo}
\end{align}
Therefore for any integer $k\geq 1$, real numbers $z\geq 1$, and $x\geq 1$  we have
\begin{align}
&\sum_{n\leq x}\left(\omega_{z}(n;m,a)-\frac{1}{\varphi(m)}\log\log z\right)^{k}\nonumber \\
&=\sum_{j=0}^{k-2}\binom{k}{j}\left(b(m,a)+O_{m}\left(\frac{1}{\log z}\right)\right)^{k-j}\sum_{n\leq x}\left(\sum_{\begin{subarray}{c}p\leq z\\ p\equiv a\, (m)\end{subarray}}f_{p}(n)\right)^{j}\nonumber \\
&\quad +k b(m,a)\sum_{n\leq x}\left(\sum_{\begin{subarray}{c}p\leq z\\ p\equiv a\, (m)\end{subarray}}f_{p}(n)\right)^{k-1} 
+O_{m}\left(\frac{k}{\log z}\right)\sum_{n\leq x}\left(\sum_{\begin{subarray}{c}p\leq z\\ p\equiv a\, (m)\end{subarray}}f_{p}(n)\right)^{k-1} \nonumber \\
&\quad +\sum_{n\leq x}\left(\sum_{\begin{subarray}{c}p\leq z\\ p\equiv a\, (m)\end{subarray}}f_{p}(n)\right)^{k}. \label{Maki-Horikita-Umeko}
\end{align}
To prove Theorem \ref{Natsuki} we examine the right-hand side of (\ref{Maki-Horikita-Umeko}) by Lemma \ref{lemma-sakai}.

\begin{proof}[Proof of Theorem \ref{Natsuki}]
Let $x\geq 1$ and $z\geq 1$ be sufficiently large. Further, we add the assumptions that 
\begin{align*}
2\leq k\leq \left(\frac{1}{\varphi(m)}\log\log z\right)^{1/3}\quad \text{and} \quad z\leq mx^{1/k}.
\end{align*}
\noindent{\bf (I)} Let $k\geq 3$ be odd. By (b) of Lemma \ref{lemma-sakai} we have in (\ref{Maki-Horikita-Umeko}) 
\begin{align}
&\sum_{n\leq x}\left(\sum_{\begin{subarray}{c}p\leq z\\ p\equiv a\, (m)\end{subarray}}f_{p}(n)\right)^{k}\nonumber \\
&=\tilde{C}_{k}x\left(\frac{1}{\varphi(m)}\log\log z\right)^{\frac{k-1}{2}}\nonumber \\
&\quad +O_{m}\left((k^{3/2}C_{k})x\left(\frac{1}{\varphi(m)}\log\log z\right)^{\frac{k-1}{2}}\frac{k^{3}}{\frac{1}{\varphi(m)}\log\log z}\right).\label{nishiki-3-1}
\end{align}
By (a) of Lemma \ref{lemma-sakai} we have
\begin{align}
&kb(m,a)\sum_{n\leq x}\left(\sum_{\begin{subarray}{c}p\leq z\\ p\equiv a\, (m)\end{subarray}}f_{p}(n)\right)^{k-1}\nonumber \\
&=b(m,a)\left(k^{3/2}C_{k}\right)x\left(\frac{1}{\varphi(m)}\log\log z\right)^{\frac{k-1}{2}}\frac{C_{k-1}}{k^{1/2}C_{k}}\nonumber \\
&\quad +O_{m}\left(\left(k^{3/2}C_{k}\right)x\left(\frac{1}{\varphi(m)}\log\log z\right)^{\frac{k-1}{2}}\frac{k^{2}}{\frac{1}{\varphi(m)}\log\log z}\right) \nonumber\\
&\quad +O_{m}\left(\left(k^{3/2}C_{k}\right)x\frac{1}{k}\right). \label{nishiki-3-2}
\end{align}
Also, we get
\begin{align}
& O_{m}\left(\frac{k}{\log z}\right)\sum_{n\leq x}\left(\sum_{\begin{subarray}{c}p\leq z\\ p\equiv a\, (m)\end{subarray}}f_{p}(n)\right)^{k-1}\nonumber\\
& \ll_{m} \left(k^{3/2}C_{k}\right)x \left(\frac{1}{\varphi(m)}\log\log z\right)^{\frac{k-1}{2}}\frac{1}{k\log z}. \label{nishiki-3-3}
\end{align}
By Corollary \ref{kei-lemma-sakai} we observe that
\begin{align}
&\sum_{j=0}^{k-2}\binom{k}{j}\left(b(m,a)+O_{m}\left(\frac{1}{\log z}\right)\right)^{k-j}\sum_{n\leq x}\left(\sum_{\begin{subarray}{c}p\leq z\\ p\equiv a\, (m)\end{subarray}}f_{p}(n)\right)^{j}\nonumber \\
&\ll \sum_{j=0}^{k-2}\binom{k}{j}M_{m}^{k-j} C_{j}x\left(\frac{1}{\varphi(m)}\log\log z\right)^{\frac{j}{2}}\nonumber \\
&=C_{k}x\left(\frac{1}{\varphi(m)}\log\log z\right)^{\frac{k-2}{2}}\sum_{j=0}^{k-2}\binom{k}{j}M_{m}^{k-j}\frac{C_{j}}{C_{k}}\frac{1}{\left(\frac{1}{\varphi(m)}\log\log z\right)^{\frac{k-2}{2}-\frac{j}{2}}}\nonumber \\
&\asymp C_{k}x\left(\frac{1}{\varphi(m)}\log\log z\right)^{\frac{k-2}{2}}\times \nonumber \\
&\quad\quad\quad\quad \times \sum_{j=0}^{k-2}\binom{k}{j}M_{m}^{k-j}
         \left(\frac{j}{e}\right)^{\frac{j}{2}}\left(\frac{k}{e}\right)^{-\frac{k}{2}}
         \frac{1}{\left(\frac{1}{\varphi(m)}\log\log z\right)^{\frac{k-2}{2}-\frac{j}{2}}}\nonumber \\
&\ll_{m} C_{k}x\left(\frac{1}{\varphi(m)}\log\log z\right)^{\frac{k-2}{2}}k\sum_{j=0}^{k-2}\binom{k-2}{j}\left(\frac{M_{m}(e/k)^{1/2}}{\left(\frac{1}{\varphi(m)}\log\log z\right)^{1/2}}\right)^{(k-2)-j}\nonumber \\
&=O_{m}\left((k^{3/2}C_{k})x\left(\frac{1}{\varphi(m)}\log\log z\right)^{\frac{k-2}{2}}\frac{1}{\sqrt{k}}\right). \label{nishiki-3-4}
\end{align}
Combining results (\ref{Maki-Horikita-Umeko}), (\ref{nishiki-3-1}), (\ref{nishiki-3-2}),(\ref{nishiki-3-3}), and  (\ref{nishiki-3-4}) we obtain that
\begin{align}
& \sum_{n\leq x}\left(\omega_{z}(n;m,a)-\frac{1}{\varphi(m)}\log\log z\right)^{k}\nonumber\\
&=\tilde{C}_{k}x\left(\frac{1}{\varphi(m)}\log\log z\right)^{\frac{k-1}{2}} +b(m,a)(k^{3/2}C_{k})x\left(\frac{1}{\varphi(m)}\log\log z\right)^{\frac{k-1}{2}}\frac{C_{k-1}}{k^{1/2}C_{k}}\nonumber\\
&\quad +O_{m}\left((k^{3/2}C_{k})\left(\frac{1}{\varphi(m)}\log\log z\right)^{\frac{k-1}{2}}\times\right. \nonumber\\
&\quad\quad\quad\quad\quad \left. \times \max\left(\frac{k^{3}}{\frac{1}{\varphi(m)}\log\log z}, \frac{1}{k^{1/2}\left(\frac{1}{\varphi(m)}\log\log z\right)^{1/2}}\right)\right). \label{ryo-tonari}
\end{align}
We shall note that
$\frac{k^{3}}{\frac{1}{\varphi(m)}\log\log z}\geq \frac{1}{k^{1/2}\left(\frac{1}{\varphi(m)}\log\log z\right)^{1/2}}$
is equivalent to 
$k\leq \left(\frac{1}{\varphi(m)}\log\log z\right)^{1/7}$. 
For $3\leq k\leq \left(\frac{1}{\varphi(m)}\log\log z\right)^{1/7}$,
we see that
\begin{align*}
\frac{C_{k-1}}{C_{k}} \asymp \frac{1}{\sqrt{k}} \gg \frac{1}{\left(\frac{1}{\varphi(m)}\log\log z\right)^{1/14}}\geq \frac{1}{\left(\frac{1}{\varphi(m)}\log\log z\right)^{1/2}}.
\end{align*}
Therefore, from (\ref{ryo-tonari}) we reach the assertion (i) in (b) of Theorem \ref{Natsuki}.

Next, let $\left(\frac{1}{\varphi(m)}\log\log z\right)^{1/7}<k\leq \left(\frac{1}{\varphi(m)}\log\log z\right)^{1/3}$. In this case, we note that
$\frac{k^{3}}{\frac{1}{\varphi(m)}\log\log z}\geq \frac{1}{k}$ is equivalent to $k\geq \left(\frac{1}{\varphi(m)}\log\log z\right)^{1/4}$. Hence,
we obtain (ii), (iii) in (b) of Theorem \ref{Natsuki}.

\bigskip

\noindent{\bf (II)} Let $k\geq 2$ be even. In (\ref{Maki-Horikita-Umeko}), first we use (a) of Lemma \ref{lemma-sakai} to obtain
\begin{align}
&\sum_{n\leq x}\left(\sum_{\begin{subarray}{c}p\leq z\\ p\equiv a\, (m)\end{subarray}}f_{p}(n)\right)^{k}\nonumber\\
&=C_{k}x\left(\frac{1}{\varphi(m)}\log\log z\right)^{\frac{k}{2}}+O_{m}\left(C_{k}x\left(\frac{1}{\varphi(m)}\log\log z\right)^{\frac{k}{2}-1}k^{3}\right). \label{nagoya-1}
\end{align}
By (b) of Lemma \ref{lemma-sakai} (the upper bound), we have
\begin{align}
&\left(kb(m,a)+O_{m}\left(\frac{k}{\log z}\right)\right)\sum_{n\leq x}\left(\sum_{\begin{subarray}{c}p\leq z\\ p\equiv a\, (m)\end{subarray}}f_{p}(n)\right)^{k-1} \nonumber \\
&\ll_{m} k \tilde{C}_{k-1}x \left(\frac{1}{\varphi(m)}\log\log z\right)^{\frac{k-2}{2}}\ll_{m} C_{k}x \left(\frac{1}{\varphi(m)}\log\log z\right)^{\frac{k}{2}-1}k^{2}. \label{nagoya-2}
\end{align}
And, by Corollary \ref{kei-lemma-sakai} we observe that
\begin{align}
& \sum_{j=0}^{k-2}\binom{k}{j}\left(b(m,a)+O_{m}\left(\frac{1}{\log z}\right)\right)^{k-j}\sum_{n\leq x}\left(\sum_{\begin{subarray}{c}p\leq z\\ p\equiv a\, (m)\end{subarray}}f_{p}(n)\right)^{j}\nonumber \\
& \ll C_{k}x \left(\frac{1}{\varphi(m)}\log\log z\right)^{\frac{k-2}{2}}\frac{e}{k}\times \nonumber \\
&\quad\quad  \times \sum_{j=0}^{k-2}\binom{k}{j}M_{m}^{2}M_{m}^{k-2-j}
\left(\left(\frac{e}{k}\right)^{1/2}\right)^{k-2-j}\frac{1}{\left(\left(\frac{1}{\varphi(m)}\log\log z\right)^{1/2}\right)^{k-j}} \nonumber \\
& \ll_{m} C_{k}x \left(\frac{1}{\varphi(m)}\log\log z\right)^{\frac{k-2}{2}}k \sum_{j=0}^{k-2}\binom{k-2}{j}
\left(\frac{M_{m}\left(\frac{e}{k}\right)^{1/2}}{\left(\frac{1}{\varphi(m)}\log\log z\right)^{1/2}}\right)^{k-2-j} \nonumber\\
&\ll_{m} C_{k}x\left(\frac{1}{\varphi(m)}\log\log z\right)^{\frac{k-2}{2}}k. \label{nagoya-3} 
\end{align}
Combining (\ref{Maki-Horikita-Umeko}), (\ref{nagoya-1})--(\ref{nagoya-3}) we finally obtain the assertion (a) of Theorem \ref{Natsuki}.
\end{proof}
%
%
%
%
Although the integer $k$ is not fixed in the above, we shall now reconsider the assertion in Theorem \ref{Natsuki} for any fixed integer $k\geq 2$.
We follow the argument in \cite{MST} and prove the next theorem.

%
%
\begin{theorem}\label{kurushi-1}
Keep the notation the above. Let $k\geq 2$ be a fixed integer. Moreover, let $m\geq 1$ be a fixed and $a$ an integer satisfying $(a,m)=1$.
For sufficiently large $x$ and $z\geq 1$ with the restriction $z\leq mx^{1/k}$, we have
\begin{align}
&\sum_{n\leq x}\left(\omega_{z}(n;m,a)-\frac{1}{\varphi(m)}\log\log z\right)^{k}\nonumber\\
&=x\sum_{j=0}^{\left[\frac{k}{2}\right]}a_{j}\left(\frac{\log\log z}{\varphi(m)}\right)^{j}
  +\begin{cases} O\left(\frac{x(\log\log z)^{\frac{k-2}{2}}}{\log z}\right) & (k\geq 2, even)\\
                 O\left(\frac{x(\log\log z)^{\frac{k-1}{2}}}{\log z}\right) & (k\geq 3, odd) 
\end{cases}, \label{utsu-wa}
\end{align}
where
\begin{align*}
a_{\left[\frac{k}{2}\right]}
=
\begin{cases}
C_{k} & (k\geq 2, even)\\
\tilde{C}_{k}+kb(m,a)C_{k-1} & (k\geq 3, odd)
\end{cases}.
\end{align*}
\end{theorem}

To prove this we shall introduce a function and some constants related to the function $G(r)$ in (\ref{teigi-Gr}).

\begin{definition}
For any non-empty finite set $A:=\{\alpha_{1},\ldots, \alpha_{l}\}$ of integers $\alpha_{i} \geq 2$, any prime $p$, and the function $G(r)$ in (\ref{teigi-Gr}),
we define
\begin{align*}
G(p,A):=G\left(p^{\alpha_{1}}\right) \cdots G\left(p^{\alpha_{l}}\right).
\end{align*} 
Let $m\geq 1$ be a fixed integer, and $a$ an intger satisfying $(a,m)=1$. Recall the formula (\ref{mertens-ma}). We see that
\begin{align}\label{gpa-wa-34}
\sum_{\begin{subarray}{c}p\leq z\\ p \equiv a\, (m)\end{subarray}}G(p,A)
=\begin{cases} \frac{1}{\varphi(m)} \log\log z +B(\alpha_{1};m,a)+O\left(\frac{1}{\log z}\right) & (A=\{\alpha_{1}\}),\\
               D_{m,a}(A)+O\left(\frac{1}{z^{\sharp A -1}}\right) & (\sharp A\geq 2), 
\end{cases}
\end{align}
where $B(\alpha; m,a)$ and $D_{m,a}(A)$ are constants. 
For simplicity we might write $D_{m,a}(A)=D_{m,a}(\alpha_{1}, \ldots, \alpha_{l})$.
\end{definition}

The following proposition is a key to derive Theorem \ref{kurushi-1}.
\begin{proposition}\label{ai-ni-makeruna}
Let $k\geq 2$ be a fixed integer and keep $(a,m)=1$, where $m\geq 1$ is a fixed integer.
For $L (\geq 1)$ many finite sets $A_{1},\ldots, A_{L}$ those are set of integers $\geq 2$, we have for sufficiently large $z\geq e^{e}$
\begin{align*}
&\sum_{\begin{subarray}{c}p_{1},\ldots, p_{L}\leq z\\ distinct\\ p_{i}\equiv a\, (m)\end{subarray}}G(p_{1}, A_{1})\cdots G(p_{L},A_{L})\\
&=\sum_{j=0}^{L}\mathcal{L}_{j}(m,a)\left(\frac{\log\log z}{\varphi(m)}\right)^{j} +O\left(\frac{(\log\log z)^{L-1}}{\log z}\right),
\end{align*}
where $\mathcal{L}_{j}(m,a)$ are constants. Especially, in the case $\sharp A_{1}=\cdots = \sharp A_{L}=1$, we have $\mathcal{L}_{L}(m,a)=1$.
\end{proposition}
\begin{proof}
It is proved by induction without difficulty as in \cite{MST}.
\end{proof}

To prove Theorem \ref{kurushi-1} we shall go back to (\ref{ao-blue}), (\ref{Gr-100}), and investigate an $I_{s}(l;m,a)$ in (\ref{Gr-100}).
Recall some notations in \cite{MST} for this purpose.

\begin{definition}\label{nls-definition-kyo}
As for $I_{s}(l;m,a)$ in (\ref{Gr-100}), put
\begin{align*}
N(l,s):=\{(\alpha_{1},\ldots, \alpha_{s})\, |\, \alpha_{1}+\cdots +\alpha_{s}=l,\ \alpha_{i}\geq 2\}.
\end{align*}
Since $s\leq k/2$, it is not empty. We shall define an equivalent relation $\simeq_{1}$ on the set $N(l,s)$.
If $(\alpha_{1},\ldots, \alpha_{s})$ and $(\alpha_{1}^{\prime}, \ldots, \alpha_{s}^{\prime})\in N(l,s)$ satisfy
$\{\alpha_{1},\ldots, \alpha_{s}\}=\{\alpha_{1}^{\prime}, \ldots, \alpha_{s}^{\prime}\}$, then we write
$(\alpha_{1},\ldots, \alpha_{s})\simeq_{1} (\alpha_{1}^{\prime}, \ldots, \alpha_{s}^{\prime})$. And we represent the partition of $N(l,s)$
by this relation $\simeq_{1}$ as
\begin{align}\label{parti-Nls-303}
N(l,s)=\bigsqcup_{j=1}^{n(l,s)}N_{j}(l,s).
\end{align}
We shall write for each class $N_{j}(l,s)$ in the partition
\begin{align*}
c(N_{j}(l,s)):=\frac{l!}{\alpha_{1}!\cdots \alpha_{s}!}.
\end{align*}
Moreover, for a representative $(\alpha_{1},\ldots, \alpha_{s})$ of a class $N_{j}(l,s)$,
we denote by $\beta_{1},\ldots, \beta_{d}$ all distinct $\alpha_{i}$ and by $m(\beta_{j})$ the multiplicity of $\beta_{j}$, that is,
\begin{align*}
m(\beta_{j}):=\sum_{\begin{subarray}{c}\alpha_{i}\\ \alpha_{i}=\beta_{j}\\ (\alpha_{1},\ldots, \alpha_{s})\\ (\textit{a representative of $N_{j}(l,s)$})\end{subarray}}1
\end{align*}
\end{definition}

Using these $\beta_{1},\ldots, \beta_{d}$ of a class $N_{j}(l,s)$ we introduce an equivalent relation on $S_{s}$
which is the set of the permutations of $\{1,2,\ldots, s\}$.

\begin{definition}\label{AYATAKA}
Let $\beta_{1},\ldots, \beta_{d}$ be the integers for a representative $(\alpha_{1},\ldots, \alpha_{s})$ of a class $N_{j}(l,s)$ in (\ref{parti-Nls-303}).
We put labels $i(u,u^{\prime})$ on indexes $i$ of $\alpha_{i}$ as
\begin{align*}
\beta_{1}&=\alpha_{i(1,1)}, \alpha_{i(1,2)},\ldots, \alpha_{i(1,m(\beta_{1}))},\\
\beta_{2}&=\alpha_{i(2,1)}, \alpha_{i(2,2)},\ldots, \beta_{i(2,m(\beta_{2}))},\\
         &\cdots\\
\beta_{d}&=\alpha_{i(d,1)}, \alpha_{i(d,2)}, \ldots, \beta_{i(d,m(\beta_{d}))}. 
\end{align*}
Next, for any $\beta_{j}$ ($j=1,\ldots, d$) and $\sigma \in S_{s}$ (the set of the permutations of $\{1,\ldots, s\}$) we set
\begin{align*}
L(\sigma, \beta_{j}):=\{\nu\,|\, \nu=\sigma^{-1}(i(j,1)), \sigma^{-1}(i(j,2)),\ldots, \sigma^{-1}(j,m(\beta_{j}))\}.
\end{align*}
Here, we define an equivalent relation $\simeq_{2}$ on the set $S_{s}$. If $\sigma$ and $\sigma^{\prime}\in S_{s}$ satisfy that
\begin{align*}
L(\sigma, \beta_{1})=L(\sigma^{\prime}, \beta_{1}), \ldots, L(\sigma, \beta_{d})=L(\sigma^{\prime}, \beta_{d}),
\end{align*}
then we write $\sigma \simeq_{2}\sigma^{\prime}$. Since the number of equivalent classes by $\simeq_{2}$ is $\frac{s!}{m(\beta_{1})!\cdots m(\beta_{d})!}$,
we shall represent the partition of $S_{s}$ as
\begin{align*}
S_{s}=\bigsqcup_{\iota =1}^{\frac{s!}{m(\beta_{1})!\cdots m(\beta_{d})!}}T_{\iota}(\beta_{1},\ldots, \beta_{d}; N_{j}(l,s)).
\end{align*}
Moreover, we shall express the complete system of representative of $S_{s}$ by $\simeq_{2}$ as
\begin{align}
T(\beta_{1}, \ldots, \beta_{d}; N_{j}(l,s))
=\left\{ \tau_{1}, \ldots, \tau_{\frac{s!}{m(\beta_{1})!\cdots m(\beta_{d})!}}\right\}. \label{parti-S_s-100}
\end{align}
\end{definition}

We shall prove Theorem \ref{kurushi-1}, briefly, since the process of the proof is similar to the argument in \cite{MST},
we may omit details.

\begin{proof}[Proof of Theorem \ref{kurushi-1}]
First, by (\ref{cococo}) and Lemma \ref{lemma-sakai} we observe that
\begin{align}
&\sum_{n\leq x}\left(\omega_{z}(n;m,a)-\frac{1}{\varphi(m)}\log\log z\right)^{k} \nonumber\\
&=\sum_{l=2}^{k}\binom{k}{l}b(m,a)^{k-l}\sum_{n\leq x}\left(\sum_{\begin{subarray}{c}p\leq z\\ p\equiv a\, (m)\end{subarray}}f_{p}(n)\right)^{l} +b(m,a)^{k}x \nonumber\\
&\quad +\begin{cases} O\left(\frac{x(\log\log z)^{\frac{k-2}{2}}}{\log z}\right) & (k\geq 2, even)\\
                      O\left(\frac{x(\log\log z)^{\frac{k-1}{2}}}{\log z}\right) & (k\geq 3, odd)
\end{cases}. \label{kocha-cookies}
\end{align}
Here, by (\ref{ao-blue}) we get
\begin{align}
\sum_{n\leq x}\left(\sum_{\begin{subarray}{c}p\leq z p\equiv a\, (m)\end{subarray}}f_{p}(n)\right)^{l}
=x\sum_{\begin{subarray}{c}p_{1},\ldots, p_{l}\leq z \\ p_{i}\equiv a\, (m) \\ p_{1}\cdots p_{l}: \textit{square-full}\end{subarray}}G(p_{1}\cdots p_{l}) 
+O\left(\frac{x}{\log z}\right). \label{gisu-gisu-100}
\end{align}
As in (\ref{Gr-100}) for the sum of $G(\cdot)$ we shall write
\begin{align}\label{suna-kuzureru}
\sum_{\begin{subarray}{c}p_{1},\ldots, p_{l}\leq z \\ p_{i}\equiv a\, (m) \\ p_{1}\cdots p_{l}: \textit{square-full}\end{subarray}}G(p_{1}\cdots p_{l})
=:\sum_{s\leq l/2}I_{s}(l;m,a).
\end{align}
In each $I_{s}(l;m,a)$ in  (\ref{suna-kuzureru}) we note that
\begin{align}\label{suna-kuzureru-2}
I_{s}(l;m,a)
=\sum_{j=1}^{n(l,s)}c(N_{j}(l,s))\sum_{\forall (\alpha_{1},\ldots, \alpha_{s})\in N_{j}(l,s)}
\sum_{\begin{subarray}{c}q_{1}<\cdots <q_{s}\leq z\\ q_{i}\equiv a\, (m)\end{subarray}}G\left(q_{1}^{\alpha_{1}}\right)\cdots G\left(q_{s}^{\alpha_{s}}\right)
\end{align}
by Definition \ref{nls-definition-kyo}.

We now take a class $N_{j}(l,s)$ in each $I_{s}(l;m,a)$ in (\ref{suna-kuzureru}), where the class is defined by (\ref{parti-Nls-303}) 
(in Definition \ref{nls-definition-kyo}), and we choose a representative $(\alpha_{1},\ldots, \alpha_{s})$ of the class $N_{j}(l,s)$.
Moreover, let $S_{s}$ be the set of the permutations of $\{1,\ldots, s\}$, and $T(\beta_{1},\ldots, \beta_{d};N_{j}(l,s))$ be the complete system
of the representative of $S_{s}$ defined by (\ref{parti-S_s-100}) (in Definition \ref{AYATAKA}).
Then, we observe that in $I_{s}(l;m,a)$ (in (\ref{suna-kuzureru}))
\begin{align}
&\sum_{\begin{subarray}{c}q_{1},\ldots, q_{s}\leq z\\ \textit{distinct}\\ q_{i}\equiv a\, (m)\\
((\alpha_{1},\ldots, \alpha_{s}); \textit{a representative of $N_{j}(l,s)$})\end{subarray}}G\left(q_{1}^{\alpha_{1}}\right)\cdots G\left(q_{s}^{\alpha_{s}}\right)\nonumber\\
&= \sum_{\sigma \in S_{s}}
   \sum_{\begin{subarray}{c}q_{\sigma(1)}<\cdots <q_{\sigma(s)}\leq z\\ q_{i}\equiv a\, (m)\end{subarray}}
   G\left(q_{1}^{\alpha_{1}}\right)\cdots G\left(q_{s}^{\alpha_{s}}\right) \nonumber \\
       &=m(\beta_{1})!\cdots m(\beta_{d})! \sum_{\tau \in T(\beta_{1},\ldots, \beta_{d}; N_{j}(l,s))} 
                                            \sum_{\begin{subarray}{c}q_{\sigma(1)}<\cdots <q_{\sigma(s)}\leq z\\ q_{i}\equiv a\, (m)\end{subarray}} 
                                             G\left(q_{1}^{\alpha_{\tau(1)}}\right)\cdots G\left(q_{s}^{\alpha_{\tau(s)}}\right)\nonumber \\
       &=m(\beta_{1})!\cdots m(\beta_{d})! 
         \sum_{\forall (\alpha_{1},\ldots, \alpha_{s})\in N_{j}(l,s)}
         \sum_{\begin{subarray}{c}q_{\sigma(1)}<\cdots <q_{\sigma(s)}\leq z\\ q_{i}\equiv a\, (m)\end{subarray}}
         G\left(q_{1}^{\alpha_{1}}\right)\cdots G(q_{s}^{\alpha_{s}}). \label{Harumi-Mikan}
\end{align}
Using (\ref{Harumi-Mikan}) in (\ref{suna-kuzureru-2}) we get
\begin{align*}
I_{s}(l;m,a)=\sum_{j=1}^{n(l,s)}\frac{c(N_{j}(l,s))}{m(\beta_{1})!\cdots m(\beta_{d})!}
              \hspace{-1.5cm}\sum_{\begin{subarray}{c}q_{1},\ldots, q_{s}\leq z\\ \textit{distinct}\\ q_{i}\equiv a\, (m)\\
((\alpha_{1},\ldots, \alpha_{s}); \textit{a representative of $N_{j}(l,s)$})\end{subarray}}\hspace{-1.5cm} G\left(q_{1}^{\alpha_{1}}\right)\cdots G\left(q_{s}^{\alpha_{s}}\right).
\end{align*}
Next, we apply Proposition \ref{ai-ni-makeruna} to the above, then we obtain
\begin{align}
I_{s}(l;m,a)=\sum_{J=0}^{s}a_{J}^{(l)}\left(\frac{\log\log z}{\varphi(m)}\right)^{J}+O\left(\frac{(\log\log z)^{s-1}}{\log z}\right), \label{FAMIMA33}
\end{align}
where for $J=s$
\begin{align*}
a_{s}^{(l)}=\sum_{j=1}^{n(l,s)}\frac{c(N_{j}(l,s))}{m(\beta_{1})!\cdots m(\beta_{d})!}.
\end{align*}
We use (\ref{FAMIMA33}) in (\ref{suna-kuzureru}), then
\begin{align*}
\sum_{\begin{subarray}{c}p_{1},\ldots, p_{l}\leq z \\ p_{i}\equiv a\, (m) \\ p_{1}\cdots p_{l}: \textit{square-full}\end{subarray}}G(p_{1}\cdots p_{l})
=\sum_{J=0}^{\left[\frac{l}{2}\right]}a_{J}^{\prime(l)}\left(\frac{\log\log z}{\varphi(m)}\right)^{J} +O\left(\frac{(\log\log z)^{\left[\frac{l}{2}\right]-1}}{\log z}\right).
\end{align*} 
Moreover, we apply this to (\ref{gisu-gisu-100}), and we use the formula for $\sum_{n\leq x}\left(\sum_{p\leq z}f_{p}(n)\right)^{l}$ which is obtained by the process
to (\ref{kocha-cookies}), then we reach the assertion (\ref{utsu-wa}) of Theorem \ref{kurushi-1}.

If $k\geq 2$ is even in (\ref{utsu-wa}), we see that
\begin{align*}
a_{\left[\frac{k}{2}\right]}=a_{\frac{k}{2}}
=\sum_{j=1}^{n(k,k/2)}\frac{c\left(N_{j}\left(k,\frac{k}{2}\right)\right)}{m(\beta_{1})!\cdots m(\beta_{d})!}
=\frac{c\left(N\left(k,\frac{k}{2}\right)\right)}{\left(\frac{k}{2}\right)!}=\frac{k!}{2^{k/2}\left(\frac{k}{2}\right)!}=C_{k}.
\end{align*}
If $k\geq 3$ is odd in (\ref{utsu-wa}), we observe that
\begin{align*}
&a_{\left[\frac{k}{2}\right]}=a_{\frac{k}{2}-\frac{1}{2}}\\
                            &=\sum_{j=1}^{n\left(k, \frac{k-1}{2}\right)}\frac{c\left(N_{j}(k,\frac{k-1}{2})\right)}{m(\beta_{1})!\cdots m(\beta_{d})!}
                              +kb(m,a)\sum_{j=1}^{n\left(k-1,\frac{k-1}{2}\right)}\frac{c\left(N_{j}\left(k-1,\frac{k-1}{2}\right)\right)}{m(\beta_{1})!\cdots m(\beta_{d})!}\\
&=\frac{c\left(N \left(k,\frac{k-1}{2}\right)\right)}{1!\left(\frac{k-3}{2}\right)!} +kb(m,a) C_{k-1}\\
&=\frac{1}{\left(\frac{k-3}{2}\right)!}\frac{k!}{3!2^{\frac{k-3}{2}}} +kb(m,a)C_{k-1}\\
&=\tilde{C}_{k}+kb(m,a)C_{k-1}. 
\end{align*}
We complete the proof of Theorem \ref{kurushi-1}.
\end{proof}

In the cases $k=2,3,\ldots, 6$ for Theorem \ref{kurushi-1}, we can obtain explicit asymptotic formulas.
\begin{theorem}\label{okura-1}
Let $m\geq 1$ be a fixed integer and $(a,m)=1$. Keep the notation as in the above. For sufficiently large $x\geq 1$ and $z\geq 1$,
we have the following formulas.
\begin{enumerate}
\item[\rm (a)] If $z\leq mx^{1/2}$, then
\begin{align*}
& \sum_{n\leq x}\left(\omega_{z}(n;m,a)-\frac{1}{\varphi(m)}\log\log z\right)^{2}\\
&=x\left(\frac{\log\log z}{\varphi(m)}\right)+ \left(b(m,a)^{2}+B(2;m,a)\right)x +O\left(\frac{x}{\log z}\right).
\end{align*}
\item[\rm (b)] If $z\leq mx^{1/3}$, then
\begin{align*}
&  \sum_{n\leq x}\left(\omega_{z}(n;m,a)-\frac{1}{\varphi(m)}\log\log z\right)^{3}\\
&= (1+3b(m,a)) x \left(\frac{\log\log z}{\varphi(m)}\right)\\
&\quad +\left(b(m,a)^{3} +3b(m,a)B(2;m,a)+B(3;m,a)\right)x +O\left(\frac{x\log\log z}{\log z}\right).
\end{align*}
\item[\rm (c)] If $z\leq mx^{1/4}$, then
\begin{align*}
&  \sum_{n\leq x}\left(\omega_{z}(n;m,a)-\frac{1}{\varphi(m)}\log\log z\right)^{4}\\
&= 3x \left(\frac{\log\log z}{\varphi(m)}\right)^{2}\\
&\quad +\left(6B(2;m,a)+4b(m,a)+6b(m,a)^{2}+1\right)x\left(\frac{\log\log z}{\varphi(m)}\right)\\
&\quad + \left(3B(2;m,a)^{2}-3D_{m,a}(2,2) +B(4;m,a) \right. \\
&\quad\quad\quad   \left.    +4b(m,a)B(3;m,a)+6b(m,a)^{2}B(2;m,a)+b(m,a)^{4}\right)x\\
&\quad +O\left(\frac{x\log\log z}{\log z}\right)
\end{align*}

\item[\rm (d)] If $z\leq mx^{1/5}$, then
\begin{align*}
& \sum_{n\leq x}\left(\omega_{z}(n;m,a)-\frac{1}{\varphi(m)}\log\log z\right)^{5}\\
&= (10+15b(m,a))x\left(\frac{\log\log z}{\varphi(m)}\right)^2 \\
&\quad +\left(10(B(2;m,a)+B(3;m,a))+1\right.\\
&\quad\quad\quad   +30b(m,a)B(2;m,a)+5b(m,a) \\
&\quad\quad\quad \left. +10b(m,a)^{2}+10b(m,a)^{3} \right) x\left(\frac{\log\log z}{\varphi(m)}\right)\\
&\quad +\left(10B(2;m,a)B(3;m,a)-10D_{m,a}(3,2)+B(5;m,a)\right.\\
&\quad +15b(m,a)B(2;m,a)^{2} -15b(m,a)D_{m,a}(2,2) +5b(m,a)B(4;m,a)\\
&\quad +\left. 10b(m,a)^{2}B(3;m,a) +10b(m,a)^{3}B(2;m,a)+5b(m,a)^{5}\right) x\\
&\quad +O\left(\frac{x(\log\log z)^{2}}{\log z}\right)
\end{align*}
\item[\rm (e)] If $z\leq mx^{1/6}$, then
\begin{align*}
& \sum_{n\leq x}\left(\omega_{z}(n;m,a)-\frac{1}{\varphi(m)}\log\log z\right)^{6}\\
&= 15 x\left(\frac{\log\log z}{\varphi(m)}\right)^{3}\\
&\quad +\left(45B(2;m,a)+60b(m,a)+45b(m,a)^{2}+25\right)x\left(\frac{\log\log z}{\varphi(m)}\right)^{2}\\
&\quad +\left(45\left(B(2;m,a)^{2}-D_{m,a}(2,2)\right)\right.\\
&\quad\quad +15\left(B(4;m,a)+B(2;m,a)\right)+20B(3;m,a)+1\\
&\quad\quad  +6b(m,a)\left( 10(B(2;m,a)+B(3;m,a)) +1 \right)\\
&\quad\quad + 15b(m,a)^{2}\left(6B(2;m,a)+1\right)\\
&\quad\quad +\left. 20b(m,a)^{3} +15b(m,a)^{4} \right) x\left(\frac{\log\log z}{\varphi(m)}\right)\\
&\quad +\left(15B(2;m,a)^{3}-45\left(D_{m,a}(2,2)B(2;m,a)-D_{m,a}(2,2,2)\right)\right.\\
&\quad\quad -15D_{m,a}(2,2,2) +15\left(B(4;m,a)B(2;m,a)-D_{m,a}(4,2)\right)\\
&\quad\quad +10\left(B(3;m,a)^{2}-D_{m,a}(3,3)\right)+B(6;m,a)\\
&\quad\quad +6b(m,a)\left(10B(2;m,a)B(3;m,a)-10D_{m,a}(3,2)+B(5;m,a)\right)\\
&\quad\quad +15b(m,a)^{2} \left(3B(2;m,a)^{2}-3D_{m,a}(2,2)+B(4;m,a)\right)\\
&\quad\quad +\left. 20b(m,a)^{3} B(3;m,a) +15b(m,a)^{4}B(2;m,a) +b(m,a)^{6} \right)x\\
&\quad +O\left(\frac{x(\log\log z)^{2}}{\log z}\right). 
\end{align*}
\end{enumerate}
Here, $O$-constants depend on $m$.
\end{theorem}

To prove Theorem \ref{okura-1}
we prepare explicit asymptotic formulas for $\sum_{n\leq x}\left(\sum_{\begin{subarray}{c}p\leq z\\ p\equiv a\, (m)\end{subarray}}f_{p}(n)\right)^{k}$
($k=2,\ldots, 6$).

\begin{lemma}\label{jyokigen-yamada}
Keep the situation of Theorem \ref{okura-1}. We have the following formulas.

\begin{enumerate}
\item[\rm (a)] If $z\leq m x^{1/2}$, then
\begin{align*}
\sum_{n\leq x}\left(\sum_{\begin{subarray}{c}p\leq z\\ p\equiv a\, (m)\end{subarray}}f_{p}(n)\right)^{2}
=x\left(\frac{\log\log z}{\varphi(m)}\right) +B(2;m,a)x +O\left(\frac{x}{\log z}\right).
\end{align*}

\item[\rm (b)] If $z\leq mx^{1/3}$, then
\begin{align*}
\sum_{n\leq x}\left(\sum_{\begin{subarray}{c}p\leq z\\ p\equiv a\, (m)\end{subarray}}f_{p}(n)\right)^{3}
=x\left(\frac{\log\log z}{\varphi(m)}\right) +B(3;m,a)x +O\left(\frac{x}{\log z}\right).
\end{align*}

\item[\rm (c)] If $z\leq mx^{1/4}$, then
\begin{align*}
\sum_{n\leq x}\left(\sum_{\begin{subarray}{c}p\leq z\\ p\equiv a\, (m)\end{subarray}}f_{p}(n)\right)^{4}
&=3x\left(\frac{\log\log z}{\varphi(m)}\right)^{2}+\left(6B(2;m,a)+1\right)x\left(\frac{\log\log z}{\varphi(m)}\right)\\
&\quad +\left(3B(2;m,a)^{2}-3D_{m,a}(2,2)+B(4;m,a)\right)x\\
&\quad +O\left(\frac{x\log\log z}{\log z}\right).
\end{align*}

\item[\rm (d)] If $z\leq mx^{1/5}$, then
\begin{align*}
&\sum_{n\leq x}\left(\sum_{\begin{subarray}{c}p\leq z\\ p\equiv a\, (m)\end{subarray}}f_{p}(n)\right)^{5}\\
&=10x \left(\frac{\log\log z}{\varphi(m)}\right)^{2}\\
&\quad +\left(10(B(2;m,a)+B(3;m,a))+1\right)x\left(\frac{\log\log z}{\varphi(m)}\right)\\
&\quad +(10B(2;m,a)B(3;m,a)-10D_{m,a}(3,2)+B(5;m,a))x\\
&\quad +O\left(\frac{x\log\log z}{\log z}\right).
\end{align*}

\item[\rm (e)] If $z\leq mx^{1/6}$, then
\begin{align*}
&\sum_{n\leq x}\left(\sum_{\begin{subarray}{c}p\leq z\\ p\equiv a\, (m)\end{subarray}}f_{p}(n)\right)^{6}\\
&=15 x\left(\frac{\log\log z}{\varphi(m)}\right)^{3} +\left(45B(2;m,a)+25\right)x\left(\frac{\log\log z}{\varphi(m)}\right)^{2}\\
&\quad + \left(45\left(B(2;m,a)^{2}-D_{m,a}(2,2)\right)+15\left(B(4;m,a)+B(2;m,a)\right)\right.\\
&\quad \left. \phantom{B^{2}} +20B(3;m,a)+1 \right)x\left(\frac{\log\log z}{\varphi(m)}\right)\\
&\quad +\left(15B(2;m,a)^{3}-45\left(D_{m,a}(2,2)B(2;m,a)-D_{m,a}(2,2,2)\right)\right.\\
&\quad\quad -15D_{m,a}(2,2,2) +15\left(B(4;m,a)B(2;m,a)-D_{m,a}(4,2)\right)\\
&\quad\quad \left. +10\left(B(3;m,a)^{2}-D_{m,a}(3,3)\right)+B(6;m,a)\right)x\\
&\quad +O\left(\frac{x(\log\log z)^{2}}{\log z}\right).
\end{align*}
\end{enumerate}
Here, $O$-constants depend on $m$.
\end{lemma}
\begin{proof}
From (\ref{gisu-gisu-100}) we have
\begin{align*}
\sum_{n\leq x}\left(\sum_{\begin{subarray}{c}p\leq z\\ p\equiv a\, (m)\end{subarray}}f_{p}(n)\right)^{2}
=x\sum_{\begin{subarray}{c}q\leq z\\ q\equiv a\, (m)\end{subarray}}G\left(q^{2}\right)+O\left(\frac{x}{\log z}\right)\quad (z\leq mx^{1/2}),
\end{align*}
and
\begin{align*}
\sum_{n\leq x}\left(\sum_{\begin{subarray}{c}p\leq z\\ p\equiv a\, (m)\end{subarray}}f_{p}(n)\right)^{3}
=x\sum_{\begin{subarray}{c}q\leq z\\ q\equiv a\, (m)\end{subarray}}G\left(q^{3}\right)+O\left(\frac{x}{\log z}\right)\quad (z\leq mx^{1/3}).
\end{align*}
Here, we shall apply (\ref{gpa-wa-34}) to them. Then, we get assertions (a) and (b), at once.

In the case $z\leq mx^{1/4}$, from (\ref{gisu-gisu-100}) we observe that
\begin{align*}
&\sum_{n\leq x}\left(\sum_{\begin{subarray}{c}p\leq z\\ p\equiv a\, (m)\end{subarray}}f_{p}(n)\right)^{4}\\
&=x\left(6\sum_{\begin{subarray}{c}q_{1}<q_{2}\leq z \\ q_{i}\equiv a\, (m)\end{subarray}} G\left(q_{1}^{2}\right)
                                                                                         G\left(q_{2}^{2}\right)
         +\sum_{\begin{subarray}{c}q\leq z\\ q\equiv a\, (m)\end{subarray}}G\left(q^{4}\right) \right) +O\left(\frac{x}{\log z}\right).
\end{align*}
Here, note that
\begin{align*}
&\sum_{\begin{subarray}{c}q_{1}<q_{2}\leq z \\ q_{i}\equiv a\, (m)\end{subarray}} G\left(q_{1}^{2}\right)
                                                                                         G\left(q_{2}^{2}\right)
=
\frac{1}{2}\sum_{\begin{subarray}{c}q_{1}, q_{2}\leq z\\ \textit{distinct}\\ q_{i}\equiv a\, (m)\end{subarray}}G\left(q_{1}^{2}\right)G\left(q_{2}^{2}\right)\\
&=\frac{1}{2}\left( \left( \sum_{\begin{subarray}{c}q_{1}\leq z\\ q_{1}\equiv a\, (m)\end{subarray}}G\left(q_{1}^{2}\right)\right)
                    \left(\sum_{\begin{subarray}{c}q_{2}\leq z\\ q_{2}\equiv a\, (m)\end{subarray}}G\left(q_{2}^{2}\right)\right)
              -\sum_{\begin{subarray}{c}q\leq z\\ q\equiv a\, (m)\end{subarray}}G\left (q^{2}\right) G\left(q^{2}\right) \right).
\end{align*}
By (\ref{gpa-wa-34}) we have
\begin{align*}
\sum_{\begin{subarray}{c}q_{1}<q_{2}\leq z \\ q_{i}\equiv a\, (m)\end{subarray}} G\left(q_{1}^{2}\right)
                                                                                         G\left(q_{2}^{2}\right)
&=\frac{1}{2}\left(\frac{\log\log z}{\varphi(m)}\right)^{2} +B(2;m,a)\left(\frac{\log\log z}{\varphi(m)}\right)\\
&\quad +\frac{1}{2}\left(B(2;m,a)^{2}-D_{m,a}(2,2)\right) +O\left(\frac{\log\log z}{\log z}\right).
\end{align*}
Also, we see that
\begin{align*}
\sum_{\begin{subarray}{c}q\leq z\\ q\equiv a\, (m)\end{subarray}}G\left(q^{4}\right)=\frac{1}{\varphi(m)}\log\log z +B(4;m,a)+O\left(\frac{1}{\log z}\right).
\end{align*}
Combining these results we obtain the assertion (c).

As for the assertion (d), first we obtain that for $z\leq mx^{1/5}$
\begin{align*}
&\sum_{n\leq x}\left(\sum_{\begin{subarray}{c}p\leq z\\ p\equiv a\, (m)\end{subarray}}f_{p}(n)\right)^{5}\\
&=x\left(\frac{5!}{3!2!}\sum_{\begin{subarray}{c}q_{1}, q_{2}\leq z\\ \textit{distinct}\\ q_{i}\equiv a\, (m)\end{subarray}}G\left(q_{1}^{3}\right)G\left(q_{2}^{2}\right)
+\sum_{\begin{subarray}{c}q\leq z\\ q\equiv a\, (m) \end{subarray}}G\left(q^{5}\right)\right) +O\left(\frac{x}{\log z}\right).
\end{align*}
As for (e) ($z\leq mx^{1/6}$), first we observe that
\begin{align*}
&\sum_{n\leq x}\left(\sum_{\begin{subarray}{c}p\leq z\\ p\equiv a\, (m)\end{subarray}}f_{p}(n)\right)^{6}\\
&=90 x \frac{1}{3!}\sum_{\begin{subarray}{c}q_{1}, q_{2}, q_{3}\leq z\\ \textit{distinct}\\ q_{i}\equiv a\, (m)\end{subarray}}
       G\left(q_{1}^{2}\right)G\left(q_{2}^{2}\right) G(q_{3}^{2})
+15x \sum_{\begin{subarray}{c}q_{1},q_{2}\leq z\\ \textit{distinct}\\ q_{i}\equiv a\, (m)\end{subarray}}G\left(q_{1}^{4}\right)G\left(q_{2}^{2}\right)\\
&\quad +20 x \frac{1}{2}\sum_{\begin{subarray}{c}q_{1}, q_{2}\leq z\\ \textit{distinct}\\ q_{i}\equiv a\, (m)\end{subarray}}G\left(q_{1}^{3}\right)G\left(q_{2}^{3}\right)
+ x\sum_{\begin{subarray}{c}q\leq z\\ q\equiv a\, (m)\end{subarray}}G\left(q^{6}\right) +O\left(\frac{x}{\log z}\right).
\end{align*}
Moreover, we note that
\begin{align*}
&\sum_{\begin{subarray}{c}q_{1}, q_{2}, q_{3}\leq z\\ \textit{distinct}\\ q_{i}\equiv a\, (m)\end{subarray}}
       G\left(q_{1}^{2}\right)G\left(q_{2}^{2}\right) G(q_{3}^{2})\\
&=\left(\sum_{\begin{subarray}{c}q_{1}\leq z\\ q_{1}\equiv a\, (m)\end{subarray}}G\left(q_{1}^{2}\right)\right)
   \left(\sum_{\begin{subarray}{c}q_{2}\leq z\\ q_{2}\equiv a\, (m)\end{subarray}}G\left(q_{2}^{2}\right)\right)
   \left(\sum_{\begin{subarray}{c}q_{3}\leq z\\ q_{3}\equiv a\, (m)\end{subarray}}G\left(q_{3}^{2}\right)\right)\\
&\quad -3\sum_{\begin{subarray}{c}q_{1},q_{2}\leq z\\ \textit{distinct}\\ q_{i}\equiv a\, (m)\end{subarray}}
   \left(G\left(q_{1}^{2}\right)\right)^{2}G\left(q_{2}\right)^{2}
 -\sum_{\begin{subarray}{c}q\leq z\\ q\equiv a\, (m)\end{subarray}}\left(G\left(q^{2}\right)\right)^{3}.
\end{align*}
Finally, we shall use (\ref{gpa-wa-34}) to obtain the assertions (d) and (e).
\end{proof}

By Lemma \ref{jyokigen-yamada} we now finish this section.
\begin{proof}[Proof of Theorem \ref{okura-1}]
In (\ref{kocha-cookies}), we apply Lemma \ref{jyokigen-yamada}, then we reach the assertions (a)--(e) of Theorem \ref{okura-1}.
\end{proof}

\section{Proof of Theorem \ref{umeko}}
Finally, we shall conclude this notes by proving Theorem \ref{umeko} on $\omega(n;m,a)$ which is defined in (\ref{teigi-omega-nma}).

Let $m\geq 1$ be a fixed integer and $a$ an integer such that $(a,m)=1$. 
Moreover, let $k\geq 2$ be any integer satisfying $k\leq \left(\frac{1}{\varphi(m)}\log\log z\right)^{1/3}$,
where $x\geq 1$ and $z\geq 1$ are sufficiently large numbers. 

To derive Theorem \ref{umeko} we now set $z=mx^{1/k}$.
At first, we note a relation $\omega(n;m,a)$ and $\sum_{\begin{subarray}{c}p\leq z\\ p\equiv a\, (m)\end{subarray}}f_{p}(n)$.
For any natural numbers $n\leq x$ we have
\begin{align*}
\omega(n;m,a)&=\sum_{\begin{subarray}{c}p|n\\ p\equiv a\, (m)\\ p\leq z\end{subarray}}1+O(k)
=\sum_{\begin{subarray}{c}p|n\\ p\equiv a\, (m)\\ p\leq z\end{subarray}}\left(1-\frac{1}{p}+\frac{1}{p}\right)+O(k)\nonumber \\
&=\sum_{\begin{subarray}{c}p\leq z\\ p\equiv a\,(m)\end{subarray}}f_{p}(n)+\frac{1}{\varphi(m)}\log\log z +O_{m}(1)+O(k)\nonumber \\
&=\sum_{\begin{subarray}{c}p\leq z\\ p\equiv a\,(m)\end{subarray}}f_{p}(n)+\frac{1}{\varphi(m)}\log\log x +O_{m}(k),
\end{align*}
and, then we see that
\begin{align}
&\sum_{n\leq x}\left(\omega(n;m,a)-\frac{1}{\varphi(m)}\log\log x\right)^{k}\nonumber \\
&= \sum_{n\leq x}\left(\sum_{\begin{subarray}{c}p\leq z\\ p\equiv a\,(m)\end{subarray}}f_{p}(n)\right)^{k}
+O\left(\sum_{l=0}^{k-1}\binom{k}{l}(M_{m}k)^{k-l}\left|\sum_{n\leq x}\left(\sum_{\begin{subarray}{c}p\leq z\\ p\equiv a\,(m)\end{subarray}}f_{p}(n)\right)^{l}\right|\right),
\label{omega-moment-1}
\end{align}
where $z=mx^{1/k}$ and $M_{m}$ denotes a some positive constant. Next, using Lemma \ref{lemma-sakai} and Corollary \ref{kei-lemma-sakai}
we shall evaluate the right-hand side of (\ref{omega-moment-1}).\\

\noindent{\bf (I)} Let $k\geq 2$ be even. Recalling $z=mx^{1/k}$ we shall apply Lemma \ref{lemma-sakai} (a) to the first term in the right-hand side of (\ref{omega-moment-1}).
We obtain that
\begin{align}
&\sum_{n\leq x}\left(\sum_{\begin{subarray}{c}p\leq z\\ p\equiv a\, (m)\end{subarray}}f_{p}(n)\right)^{k}\nonumber \\
&=C_{k}x\left(\frac{1}{\varphi(m)}\log\log \left(mx^{1/k}\right)\right)^{\frac{k}{2}}\left(1+O_{m}\left(\frac{k^{3}}{\frac{1}{\varphi(m)}\log\log (mx^{1/k})}\right)\right)\nonumber \\
&\quad +O(C_{k}x)\nonumber \\
&=C_{k}x\left(\frac{1}{\varphi(m)}\log\log x\right)^{\frac{k}{2}}\left(1+O_{m}\left(\frac{k^{3}}{\frac{1}{\varphi(m)}\log\log x}\right)\right). \label{pengin-1}
\end{align}

As for the sum in the $O$-term in (\ref{omega-moment-1}), we split $l=k-1$ and $l\leq k-2$. For $l=k-1$ (which is odd), by the upper bound in Lemma \ref{lemma-sakai} (b)
we observe that
\begin{align}
&\binom{k}{k-1}(M_{m}k)\tilde{C}_{k}x \left(\frac{1}{\varphi(m)}\log\log z\right)^{\frac{k-2}{2}}\nonumber \\
&\ll M_{m} k^{3} C_{k} x \left(\frac{1}{\varphi(m)}\log\log z\right)^{\frac{k-2}{2}}
\ll_{m} C_{k}x \left(\frac{1}{\varphi(m)}\log\log x\right)^{\frac{k-2}{2}}k^{3}. \label{pengin-2}
\end{align}
For $0\leq l \leq k-2$, we shall use Corollary \ref{kei-lemma-sakai}, then
\begin{align*}
&\sum_{l=0}^{k-2}\binom{k}{l}(M_{m}k)^{k-l}\left|\sum_{n\leq x}\left(\sum_{\begin{subarray}{c}p\leq z\\ p\equiv a\, (m)\end{subarray}}f_{p}(n)\right)^{l}\right|\\
&\ll \sum_{l=0}^{k-2}\binom{k}{l}(M_{m}k)^{k-2-l+2}C_{k}\frac{C_{l}}{C_{k}}x\left(\frac{1}{\varphi(m)}\log\log z\right)^{\frac{l}{2}-\frac{k-2}{2}+\frac{k-2}{2}}\\
&=C_{k}x\left(\frac{1}{\varphi(m)}\log\log z\right)^{\frac{k-2}{2}}(M_{m}k)^{2}\sum_{l=0}^{k-2}\binom{k}{l}\frac{C_{l}}{C_{k}}
  \left(\frac{M_{m}k}{\left(\frac{1}{\varphi(m)}\log\log z\right)^{1/2}}\right)^{k-2-l}.
\end{align*}
Note that
\begin{align*}
\frac{C_{l}}{C_{k}} \asymp \left(\frac{l}{e}\right)^{\frac{l}{2}}\left(\frac{k}{e}\right)^{-\frac{k}{2}}\ll \left(\left(\frac{e}{k}\right)^{1/2}\right)^{k-2-l}\frac{e}{k}. 
\end{align*}
We observe that
\begin{align}
&\ll C_{k}x\left(\frac{1}{\varphi(m)}\log\log z\right)^{\frac{k-2}{2}}M_{m}^{2}k
    \sum_{l=0}^{k-2}\binom{k}{l}\left(\frac{M_{m}e^{1/2}k^{1/2}}{\left(\frac{1}{\varphi(m)}\log\log z\right)^{1/2}}\right)^{k-2-l}\nonumber \\
&= C_{k}x\left(\frac{1}{\varphi(m)}\log\log z\right)^{\frac{k-2}{2}}M_{m}^{2}k \sum_{s=0}^{k-2}\binom{k}{k-2-s}\left(\frac{M_{m}e^{1/2}k^{1/2}}{\left(\frac{1}{\varphi(m)}\log\log z\right)^{1/2}}\right)^{s}\nonumber \\
&\ll C_{k}x\left(\frac{1}{\varphi(m)}\log\log z\right)^{\frac{k-2}{2}}M_{m}^{2}k^{3} \sum_{s=0}^{k-2}\frac{1}{(s+2)!}\left(\frac{M_{m}e^{1/2}k^{1/2}}{\left(\frac{1}{\varphi(m)}\log\log z\right)^{1/2}}\right)^{s}\nonumber \\
&\ll_{m} C_{k}x\left(\frac{1}{\varphi(m)}\log\log x\right)^{\frac{k-2}{2}}k^{3}. \label{pengin-3}
\end{align}
Collecting (\ref{omega-moment-1})--(\ref{pengin-3}) we reach the assertion of (a) of Theorem \ref{umeko}.

\bigskip

\noindent{\bf (II)} Let $k\geq 3$ be odd. First, we shall use Corollary \ref{kei-lemma-sakai} for the $O$-term in (\ref{omega-moment-1}).
We observe that 
\begin{align}
& \sum_{l=0}^{k-1}\binom{k}{l}(M_{m}k)^{k-l}\left|\sum_{n\leq x}\left(\sum_{\begin{subarray}{c}p\leq z\\ p\equiv a\,(m)\end{subarray}}f_{p}(n)\right)^{l}\right| \nonumber \\
& \ll \sum_{l=0}^{k-1}\binom{k}{l}(M_{m}k)^{k-1-l+1}C_{k}\frac{C_{l}}{C_{k}}x\left(\frac{1}{\varphi(m)}\log\log z\right)^{\frac{l}{2}-\frac{k-1}{2}+\frac{k-1}{2}}\nonumber \\
& \ll C_{k}x\left(\frac{1}{\varphi(m)}\log\log z\right)^{\frac{k-1}{2}}M_{m}k^{3/2}\sum_{s=0}^{k-1}\frac{1}{(s+1)!}\left(\frac{M_{m}e^{1/2}k^{3/2}}{\left(\frac{1}{\varphi(m)}\log\log z\right)^{1/2}}\right)^{s} \nonumber\\
& \ll_{m} (k^{3/2}C_{k}) x \left(\frac{1}{\varphi(m)}\log\log x\right)^{\frac{k-1}{2}}. \label{pen-chan-1}
\end{align}
From this aspect we shall apply the upper bound of (b) in Lemma \ref{lemma-sakai} to the first term in the right-hand side of (\ref{omega-moment-1}).
Then, we get
\begin{align}
\sum_{n\leq x}\left(\sum_{\begin{subarray}{c}p\leq z\\ p\equiv a\, (m)\end{subarray}}f_{p}(n)\right)^{k}
&\ll_{m} \tilde{C_{k}}x \left(\frac{1}{\varphi(m)}\log\log z\right)^{\frac{k-1}{2}} \nonumber \\
&\ll_{m} k^{3/2}{C_{k}}x \left(\frac{1}{\varphi(m)}\log\log x\right)^{\frac{k-1}{2}}. \label{pen-chan-2}
\end{align}
By (\ref{omega-moment-1}), (\ref{pen-chan-1}), and (\ref{pen-chan-2}) we complete the proof of the assertion (b) of Theorem \ref{umeko}.


%

\newpage

\noindent Tokuhon Makoto Minamide\\
Graduate School of Sciences and Technology for Innovation\\
Yamaguchi University\\
Yoshida 1677-1, Yamaguchi 753-8512, Japan\\
E-mail: minamide@yamaguchi-u.ac.jp\\

\noindent Haruka Sakai\\
Graduate School of Sciences and Technology for Innovation\\
Yamaguchi University\\
Yoshida 1677-1, Yamaguchi 753-8512, Japan\\
E-mail: e003vbv@y-u.jp\\

\noindent Yoshio Tanigawa\\
Nishizato 2-13-1, Meito, Nagoya 465-0084, Japan\\
E-mail: tanigawa@math.nagoya-u.ac.jp
\end{document}